\documentclass[11pt, svgnames]{article}
\usepackage[a4paper,margin=2cm]{geometry}
\usepackage{caption}

\usepackage{amsmath,amsthm,amssymb,amsfonts,amscd}
\usepackage{enumerate}
\usepackage{todonotes}
\usepackage{graphicx,tikz}
\usepackage{array,blkarray,multirow}
\usepackage{hyperref}
\usepackage{cleveref}
\usepackage{authblk}

\newtheorem{theorem}{Theorem}[section]
\newtheorem{lemma}[theorem]{Lemma}

\newtheorem{corollary}[theorem]{Corollary}
\newtheorem{proposition}[theorem]{Proposition}

\theoremstyle{definition}
\newtheorem{definition}[theorem]{Definition}

\newtheorem{problem}[theorem]{Problem}

\theoremstyle{remark}

\newcommand{\PG}[2]{\text{PG}(#1,#2)}
\newcommand{\ex}{\mathrm{ex}}
\newcommand{\cl}{\mathrm{cl}}

\newcommand{\trivialgroup}{\mathbb{Z}_1}

\newcommand{\FM}{\mathrm{FM}}

\tikzstyle{v}=[circle, draw, solid, fill=black, inner sep=0pt, minimum width=3pt]
\tikzstyle{balanced edge}=[draw=black, thick]
\tikzstyle{unbalanced edge}=[draw=red, thick, dashed]

\title{Tur\'{a}n's theorem for Dowling geometries}
\author[1]{Rutger Campbell}

\author[2]{Donggyu Kim}

\author[3]{Jorn van der Pol}

\affil[1]{School of Computing, KAIST,
Daejeon,~South~Korea}
\affil[2]{School of Mathematics, Georgia Institute of Technology, Atlanta, USA}
\affil[3]{Department of Applied Mathematics, University of Twente, Enschede, The Netherlands}

\affil[ ]{Email: \texttt{rutger@ibs.re.kr}, \texttt{donggyu@gatech.edu}, \texttt{j.g.vanderpol@utwente.nl}}

\begin{document}
\maketitle

\begin{abstract}
    The Dowling geometry $Q_n(\Gamma)$, where $\Gamma$ is a finite group, is a matroid that generalizes the complete-graphic matroid $M(K_{n+1})$. We determine the maximum size of an $N$-free submatroid of $Q_n(\Gamma)$ for various choices of $N$, including subgeometries $Q_m(\Gamma')$, lines $U_{2,\ell}$, and graphic matroids $M(H)$. When the group $\Gamma$ is trivial and $N=M(K_t)$, this problem reduces to Tur\'{a}n's classical result in extremal graph theory. We show that when $\Gamma$ is nontrivial, a complex dependence on $\Gamma$ emerges, even when $N=M(K_4)$.
\end{abstract}

\section{Introduction}

For a graph $H$, write $\ex(n,H)$ for the \emph{Tur\'{a}n number} or \emph{extremal number} of $H$: the maximum number of edges in an $n$-vertex graph that does not contain a subgraph isomorphic to $H$. One of the earliest results in extremal combinatorics is Mantel's theorem~\cite{Mantel1910}, which states that $\ex(n,K_3) = \lfloor {n^2}/{4} \rfloor$ and that the balanced complete bipartite graph is the only graph that attains this bound.
Tur\'{a}n~\cite{Turan1941} generalized this result by determining $\ex(n,K_{t+1})$ and showing that the densest $K_{t+1}$-free graphs are balanced complete $t$-partite graphs.
Erd\H{o}s, Stone, and Simonovits~\cite{ES1946,ES1966} showed that $\ex(n,H)$ is asymptotically determined by the chromatic number $\chi(H)$ when $H$ is not bipartite, by showing that $\ex(n,H) = \frac{\chi(H)-2}{\chi(H)-1}\binom{n}{2} + o(n^2)$.

Analogous results have been obtained for finite projective geometries. For a prime power $q$ and integers $n \ge t \ge 2$, write $\ex(\PG{n-1}{q},\PG{t-1}{q})$ for the maximum number of points in a subset of $\PG{n-1}{q}$ that does not contain a copy of $\PG{t-1}{q}$. Bose and Burton~\cite{BB1966} showed that $\ex(\PG{n-1}{q},\PG{t-1}{q}) = \frac{q^n-q^{n-t+1}}{q-1}$, and that the maximum is achieved only by the set of points obtained by removing a rank-$(n-t+1)$ flat from $\PG{n-1}{q}$. More recently, Geelen and Nelson~\cite{GeelenNelson2015} proved an analogue of the Erd\H{o}s--Stone--Simonovits theorem.

Tur\'{a}n-type problems such as the above are studied in a wide range of settings -- the analogous problem for hypergraphs is famously hard, see for example the survey by Keevash~\cite{Keevash2011}.

In this paper, we generalize Tur\'{a}n's theorem to Dowling geometries. For a finite group $\Gamma$, the rank\nobreakdash-$n$ Dowling geometry $Q_n(\Gamma)$ is a matroid whose structure is determined by the multiplicative structure of $\Gamma$ (we provide a proper definition in Section~\ref{sec: preliminaries}). Dowling geometries share many properties with both projective geometries as well as graphs (in fact, when $\Gamma$ is the trivial group, $Q_n(\Gamma)$ is isomorphic to the complete-graphic matroid $M(K_{n+1})$) and are therefore a natural setting in which to study Tur\'{a}n-type problems.

We consider $\ex(Q_n(\Gamma), Q_t(\Gamma'))$ for $\Gamma'$ a subgroup of $\Gamma$ and $t \le n$, i.e.\ the maximum number of points in a $Q_t(\Gamma')$-free submatroid of $Q_n(\Gamma)$.
When $\Gamma$, and hence $\Gamma'$, is trivial, this problem is equivalent to Tur\'{a}n's theorem, so we focus on nontrivial $\Gamma$.

We first identify the extremal function for subgeometries over nontrivial subgroups.

\begin{theorem}\label{thm intro: excluding subgeometry}
    Let $n\geq t \ge 3$ and let $\Gamma'$ be a nontrivial subgroup of $\Gamma$. Then
    \begin{equation*}
        \ex(Q_n(\Gamma), Q_t(\Gamma')) = |Q_n(\Gamma)| - n + t - 1 
        = |\Gamma|\binom{n}{2} + t -1.
    \end{equation*}
\end{theorem}

The situation is more subtle when $\Gamma'$ is the trivial group, i.e.\ when $Q_t(\Gamma')$ is a complete-graphic matroid.
In this case, our results are phrased in terms of \emph{Tur\'{a}n density} for brevity. For a fixed group $\Gamma$ and matroid $N$, a standard averaging argument shows that $\ex(Q_n(\Gamma), N)/(|\Gamma| \binom{n}{2})$ is decreasing in $n$, and hence that the Tur\'{a}n density
\begin{equation*}
    \pi(\Gamma, N) = \lim_{n \to \infty} \frac{\ex(Q_n(\Gamma), N)}{|\Gamma|\binom{n}{2}}
\end{equation*}
exists.
\begin{theorem}[Tur\'{a}n's theorem for Dowling geometries]\label{thm intro: excluding large Kt}
    Let $t \ge 5$ and let $\Gamma$ be a finite group.
    Then 
    \begin{equation*}
        \pi(\Gamma, M(K_t)) = \frac{t-2}{t-1}.
    \end{equation*}
\end{theorem}

Theorem~\ref{thm intro: excluding large Kt} leaves open the cases $t=3$ and $t=4$. In these cases, $\pi(\Gamma, M(K_t))$ depends on the group $\Gamma$ as well as on $t$.

For $t=3$, we obtain a generalization of Mantel's theorem to Dowling geometries.

\begin{theorem}[Mantel's theorem for Dowling geometries]\label{thm intro: Mantel for DG}
    Let $n\ge 2$ be an integer.
    \begin{enumerate}[(i)]
        \item $\ex(Q_{n-1}(\trivialgroup),M(K_3)) = \lfloor n^2/4 \rfloor $.
        \item $\ex(Q_n(\mathbb{Z}_2), M(K_3)) = \lfloor n^2/2 \rfloor$.
        \item $\ex(Q_n(\mathbb{Z}_3),M(K_3)) = \lceil n^2/2\rceil.$
        \item $\ex(Q_n(\Gamma),M(K_3)) = 2\binom{n}{2}$ whenever $\Gamma$ is a group of order at least~4.
    \end{enumerate}
\end{theorem}

For $\Gamma = \mathbb{Z}_2$, we additionally identify all extremal matroids.

For $t=4$, the picture is more complicated.
\begin{theorem}\label{thm intro: K4}
    If $\Gamma = \mathbb{Z}_2$ or $\Gamma$ is a group that has no element of order two, then $\pi(\Gamma, M(K_4)) = \frac{2}{3}$.
\end{theorem}
The assumption on $\Gamma$ is necessary. The smallest nontrivial groups that are not covered by Theorem~\ref{thm intro: K4} are $\mathbb{Z}_2 \times \mathbb{Z}_2$ and $\mathbb{Z}_4$. We show the following.
\begin{theorem}\label{thm intro: order 4}
    $\frac{4}{7} \le \pi(\mathbb{Z}_2 \times \mathbb{Z}_2, M(K_4)) \le \frac{7}{12}$ and $\frac{8}{13} \le \pi(\mathbb{Z}_4, M(K_4)) \le \frac{2}{3}$.
\end{theorem}
Theorems~\ref{thm intro: excluding large Kt},~\ref{thm intro: Mantel for DG}, and~\ref{thm intro: order 4} imply that $\pi(\mathbb{Z}_2\times\mathbb{Z}_2, M(K_t)) < \pi(\mathbb{Z}_4,M(K_t))$ if and only if $t = 4$, which reflects the fact that $M(K_4)$ has a richer set of graphic presentations.
A triangle $M(K_3)$ admits three different graphic presentations: A balanced triangle, an unbalanced theta graph, and a (tight or loose) handcuff; see Section~\ref{sec: preliminaries} for definitions.
These presentations are too small to capture the distiction between nonidentity group elements, which yields the above two extremal numbers are equal to $2\binom{n}{2}$ if $t=3$.
In contrast, $M(K_t)$ with $t\ge 5$ does not admit any non-looped graphic presentation with fewer than $t$ vertices so that the usual Tur\'{a}n graph with each edge replaced by $|\Gamma|$ parallel edges is an almost-extremal example for $\ex(Q_n(\Gamma), M(K_t))$.
The 4-clique $M(K_4)$ is somewhat in-between. When $\Gamma$ has an element of order two, $M(K_4)$ can be realized by a non-looped double-edge triangle. This suggests that $\pi(\Gamma, M(K_4))$ may be smaller than $\frac{2}{3}$, the expected Tur\'{a}n density based on the edge-blow-up of the Tur\'{a}n graph, because the edge-blow-up of the Tur\'{a}n graph contains too many copies of a double-edge triangle. In fact, $\pi(\mathbb{Z}_2\times\mathbb{Z}_2,M(K_4)) < \pi(\mathbb{Z}_4, M(K_4)) \le \frac{2}{3}$ as we showed; note that $\pi(\mathbb{Z}_2\times\mathbb{Z}_2,M(K_4)) \le \pi(\mathbb{Z}_4, M(K_4))$ is unsurprising because $\mathbb{Z}_2 \times \mathbb{Z}_2$ has more symmetries than $\mathbb{Z}_4$.
This observation suggests that a direct analogue of the Erd\H{o}s--Stone--Simonovits and the Geelen--Nelson theorems for Dowling geometries may not exist or should be more complicated. The classical Tur\'{a}n densities for graphs depend only on the chromatic number, and those for projective-geometric settings depend only on the given field $\mathbb{F}_q$ and the critical exponent of the forbidden matroid over~$\mathbb{F}_q$.
While the critical exponent for Dowling geometries is well-defined and loosely generalizes the chromatic number\footnote{Let $H$ be a simple graph. The critical exponent of the graphic matroid $M=M(H)$ over the trivial group is the minimum integer $k$ such that $P_M(2^k)>0$, in comparison to that the chromatic number $\chi(H)$ is the minimum $k$ such that $P_M(k) > 0$, where $P_M$ denotes the characteristic polynomial of $M$. See~\cite{Whittle1989} and~{\cite[\S15.3]{Oxley2011}} for details}~\cite{Whittle1989}, it depends on the group only through its order; thus, $\pi(\Gamma, M(H))$ necessarily depends on more information than just the critical exponent of $M(H)$ over $\Gamma$.

We cannot fully explain the interplay between the rich graphic structure of $M(K_4)$ and groups that contain an element of order two, and leave the following as an open problem.

\begin{problem}
    Determine $\pi(\Gamma, M(K_4))$ for groups $\Gamma \neq \mathbb{Z}_2$ that contain an element of order~2.
\end{problem}

The upper bound in Theorem~\ref{thm intro: excluding large Kt} extends to the following Erd\H{o}s--Stone--Simonovits-type result. However, as illustrated by $\ex(Q_n(\mathbb{Z}_2\times\mathbb{Z}_2), M(K_4))$, the upper bound need not be tight.
\begin{theorem}\label{thm intro: gain ESS}
    Let $\Gamma$ be a finite group and let $H$ be a graph. Then
    \begin{equation*}
        \pi(\Gamma, M(H)) \le \frac{\chi(H)-2}{\chi(H)-1}.
    \end{equation*}
\end{theorem}

Returning to Mantel's theorem for Dowling geometries, we extend this result in a different direction.
Note that $M(K_3) \cong U_{2,3}$, so it is natural to consider the extremal number of long lines $U_{2,\ell}$ with $\ell \ge 4$. We obtain these extremal numbers and identify the extremal matroids.
\begin{theorem}\label{thm intro: long line}
    Let $n \ge 2$ and $\ell \ge 4$. Then
    \begin{equation*}
        \ex(Q_n(\Gamma), U_{2,\ell}) =
            \begin{cases}
                |Q_n(\Gamma)|           & \text{if } \ell \ge |\Gamma|+3, \\
                |Q_n(\Gamma)|-n+1       & \text{if } \ell = |\Gamma|+2, \\
                (\ell-1)\binom{n}{2}    & \text{if } \ell \le |\Gamma|+1.
            \end{cases}
    \end{equation*}
\end{theorem}

A \emph{variety} of simple matroids is a class $\mathcal{M}$ of matroids that is closed under simple minors (where we simplify after contraction) and direct sums, for which there exists a sequence $(U_n)_{n=1}^\infty$ of matroids such that $U_n \in \mathcal{M}$ and every rank-$n$ matroid in $\mathcal{M}$ is a submatroid of $U_n$. Such a sequence of matroids is called a \emph{sequence of universal models} for the variety.
Projective geometries over finite fields and Dowling geometries over finite groups are the two main examples of sequences of universal models for varieties of matroids; Kahn and Kung~\cite{KK1982} showed that, apart from those two families, there are only two additional degenerate families of matroid varieties: matchstick and origami matroids.

For the sake of completeness, we include the extremal functions for matchstick and origami matroids in Section~\ref{sec: matchstick and origami}.
Together with the Bose--Burton theorem for projective geometries, Tur\'{a}n's theorem for graphs, and Theorem~\ref{thm intro: excluding subgeometry}, this gives a complete classification of extremal numbers $\ex(U_n,U_m)$ where $U_m \subseteq U_n$ are members of a sequence of universal models for a variety of simple matroids.

\subsection*{Structure of the paper}

The remainder of the paper is structured as follows. In Section~\ref{sec: preliminaries} we review some of the terminology related to matroids, gain graphs, and Dowling geometries. In Section~\ref{sec: long lines and subgeometries} we compute the extremal numbers for long lines and Dowling geometries and prove Theorems~\ref{thm intro: excluding subgeometry} and~\ref{thm intro: long line}. We prove Theorem~\ref{thm intro: Mantel for DG} in Section~\ref{sec: Mantel}. In Section~\ref{sec: graphic}, we consider Tur\'{a}n density for graphic matroids; this is where we prove Theorems~\ref{thm intro: K4}, \ref{thm intro: order 4}, and~\ref{thm intro: gain ESS}. In Section~\ref{sec: excluding large Kt} we consider cliques larger than a triangle and prove Theorem~\ref{thm intro: excluding large Kt}. Finally, in Section~\ref{sec: matchstick and origami} we consider Tur\'{a}n problems for matchstick and origami matroids.

\section{Preliminaries}\label{sec: preliminaries}

\subsection{Matroids}

We assume that the reader is familiar with matroid theory and refer to~\cite{Oxley2011} for any undefined terminology.
We denote the number of elements of a matroid $M$ by $|M|$.
The rank-$2$ uniform matroid $U_{2,n}$ is a \emph{line}; a line is \emph{long} if $n\ge 3$ and is \emph{very long} if $n\ge 4$.
For matroids $M$ and $N$, we say that $N$ is a \emph{submatroid} (or \emph{restriction}) of $M$, denoted $N\subseteq M$, if $N$ can be obtained from $M$ by deleting a set of its elements.
For a subset $X$ of elements of $M$, we denote by $M|X$ the submatroid of $M$ obtained by deleting all elements not in $X$.
A matroid $M$ is \emph{$N$-free} if $M$ has no submatroid isomorphic to~$N$.
An injective map $\psi\colon E(N) \to E(M)$ is called an \emph{embedding} if the submatroid of $M$ restricted to the image of $\psi$ is isomorphic to $N$.

\subsection{Gain graphs and Dowling geometries}

Throughout this paper, $\Gamma$ will always denote a finite group with multiplication as its operation; we use $1$ for the identity element in $\Gamma$.
We use $\mathbb{Z}_k$ to denote the cyclic group on $k$ elements.
Notably, $\trivialgroup$ denotes the trivial group. For groups $\Gamma$ and $\Gamma'$, we use $\Gamma' \le \Gamma$ to indicate that $\Gamma'$ is a subgroup of $\Gamma$.

All graphs in this paper are finite, but allow loops and parallel edges. To avoid trivialities, all graphs will contain at least one edge. 

A \emph{$\Gamma$-gain graph} $(G,\phi)$ is a graph $G$, together with an orientation of the edges of $G$ and a \emph{gain function} $\phi\colon E \rightarrow \Gamma$. Two gain graphs $(G,\phi)$ and $(G',\phi')$ are \emph{isomorphic} if the underlying graphs $G$ and $G'$ are isomorphic and, for each edge $e$ of $G$ and the corresponding edge $e'$ of $G'$, $\phi(e) = \phi'(e')$ if the directions of $e$ and $e'$ agree and $\phi(e) = \phi'(e')^{-1}$ otherwise. We shall often simply refer to $G$ as the gain graph (leaving the gain function $\phi$ implicit) and use $\underline{G}$ for the simplification of the underlying graph.

For a cycle $C$ of $G$, with edge sequence $e_1, e_2, \ldots, e_k$, we define
\begin{equation*}
    \phi(C) = \phi(e_1)^{\sigma_1} \cdot \phi(e_2)^{\sigma_2} \cdot \ldots \cdot \phi(e_k)^{\sigma_k},
\end{equation*}
where $\sigma_i = +1$ if the direction in which $C$ traverses $e_i$ agrees with the orientation of $e_i$, and $\sigma_i = -1$ otherwise. While $\phi(C)$ depends on the starting point and orientation of $C$, whether $\phi(C) = 1$ or not is independent of both. When the cycle $C$ has at least two edges and $\phi(C) = 1$, we call it \emph{balanced}, and \emph{unbalanced} otherwise.

A biased graph $(G,\mathcal{B})$ is a graph together with a subset $\mathcal{B}$ of its cycles that satisfies the theta-property: in any theta-subgraph of $G$, if two of its cycles are in $\mathcal{B}$, then so is the third. The collection of balanced cycles in a gain graph satisfies the theta-property and therefore forms a biased graph (see~\cite{Zaslavsky1989}).

Zaslavsky~\cite{Zaslavsky1991} introduced the \emph{frame matroid} (or bias matroid) $\FM(G,\mathcal{B})$ of a biased graph $(G,\mathcal{B})$ to be the matroid on the edges of $G$ whose set of circuits comprises the edge sets of balanced cycles, two unbalanced cycles sharing a vertex, two vertex-disjoint unbalanced cycles together with a simple path connecting them, and theta-subgraphs all of whose cycles are unbalanced (such subgraphs are called balanced cycles, tight and loose handcuffs, and unbalanced thetas, see Figure~\ref{fig: handcuffs and theta}). Since every gain graph $(G,\phi)$ gives rise to a biased graph, the frame matroid $\FM(G,\phi)$ of the gain graph $(G,\phi)$ is well defined; we shall often simply write $\FM(G)$ for $\FM(G,\phi)$.

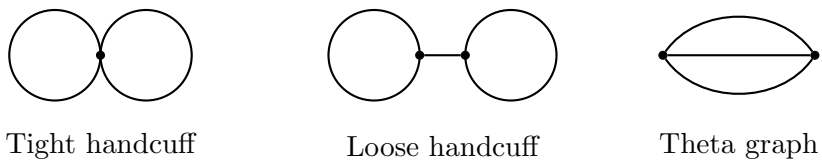
\begin{figure}[h!]
    \centering
    \begin{tikzpicture}
        \tikzstyle{v}=[circle, draw, solid, fill=black, inner sep=0pt, minimum width=3pt]
        \begin{scope}[xshift=0cm]
            \draw[thick] (0,0) circle (0.6);
            \draw[thick] (1.2,0) circle (0.6);
            \node[v] (x) at (0.6,0) {};
            \node () at (0.6,-1.2) {Tight handcuff};
        \end{scope}
        \begin{scope}[xshift=4.2cm]
            \draw[thick] (0,0) circle (0.6);
            \draw[thick] (1.8,0) circle (0.6);
            \node[v] (x) at (0.6,0) {};
            \node[v] (y) at (1.2,0) {};
            \draw[thick] (x) -- (y);
            \node () at (0.9,-1.2) {Loose handcuff};
        \end{scope}
        \begin{scope}[xshift=8cm]
            \node[v] (x) at (0,0) {};
            \node[v] (y) at (2,0) {};
            \draw[thick]
                (x) to[bend left=55](y)
                (x) to[bend left=0](y)
                (x) to[bend left=-55](y);
            \node () at (1,-1.2) {Theta graph};
        \end{scope}
    \end{tikzpicture}
    \caption{The circuits of $\FM(G,\mathcal{B})$ are the cycles in $\mathcal{B}$, or subgraphs of $G$ that are isomorphic to subdivisions of the three graphs above all of whose cycles are not in $\mathcal{B}$.}
    \label{fig: handcuffs and theta}
\end{figure}

The Dowling geometry $Q_n(\Gamma)$ was introduced by Dowling~\cite{Dowling1973a}. In this paper the following definition in terms of gain graphs will be particularly useful. Let $K_n^\Gamma$ be the gain graph on vertex set $[n]$ with, for each pair $i < j$ of vertices and each $x \in \Gamma$, an edge $x_{ij}$ oriented from $i$ to $j$ and labelled $x$; and for each vertex $i$ a loop $b_i$ whose label is independent of $\Gamma$. The Dowling geometry is the frame matroid of the gain graph so constructed, $Q_n(\Gamma) = \FM(K_n^\Gamma)$. The elements $b_i$ are called the \emph{joints} of $Q_n(\Gamma)$; they play an important role in the results in Section~\ref{sec: long lines and subgeometries}, where we discuss them further.

We study the extremal problem for Dowling geometries by reducing it to an equivalent problem on gain graphs. A complicating factor here is that in the construction of the frame matroid of a gain graph, information about gains is lost, and only balance remains. We review an operation on gain graphs that preserves their matroidal structure~\cite[Section~5]{Zaslavsky1989}.

Let $(G,\phi)$ be a $\Gamma$-gain graph, $v\in V(G)$, and $\gamma \in \Gamma$.
We say that a new labeling $\phi'$ is obtained by \emph{switching} $\phi$ by $\gamma$ at $v$ if 
\begin{equation*}
    \phi'(e) = 
    \begin{cases}
        \phi(e) \cdot \gamma & \text{if $v$ is the head of $e$}, \\
        \gamma^{-1} \cdot \phi(e) & \text{if $e$ is the tail of $e$}, \\
        \phi(e) & \text{otherwise}. \\
    \end{cases}
\end{equation*}
Two gain graphs are \emph{switching-equivalent} if they have the same underlying graph and one can be obtained from the other by applying a sequence of switching operations. Switching at a vertex does not change the balance of any cycle in $v$, so if $G$ and $H$ are switching-equivalent gain graphs, then $\FM(G) = \FM(H)$. Gain graphs $(G,\phi)$ and $(G',\phi')$ are \emph{switching-isomorphic} if $(G,\phi)$ is isomorphic to a graph that is switching-equivalent to $(G',\phi')$.

Let $H$ be a gain graph.
A gain graph $G$ has a \emph{balanced $H$-copy} if $G$ has a gain subgraph switching-isomorphic to $H$.
For a set $\mathcal{H}$ of gain graphs, a gain graph $G$ has a \emph{balanced $\mathcal{H}$-copy} if $G$ has a balanced $H$-copy for some $H\in \mathcal{H}$.

\begin{lemma}
    Let $G$ and $H$ be gain graphs.
    \begin{enumerate}[(i)]
        \item If $G$ has a balanced $H$-copy, then $\FM(G)$ has a submatroid isomorphic to $\FM(H)$.
        \item If $\FM(H) \subseteq \FM(G)$ and $H_1,\ldots,H_k$ are the all gain graphs up to switching-isomorphism such that $\FM(H_i) \cong \FM(H)$, then $G$ has a balanced $H_i$-copy for some $i$.
        \qed
    \end{enumerate}
\end{lemma}

\subsection{\texorpdfstring{$\mathbb{Z}_2$}{Z2}-gain-graphic representations of \texorpdfstring{$M(K_t)$}{M(Kt)}}

In this section, we consider the inequivalent frame representations of $M(K_t)$ over $\mathbb{Z}_2$.
For the general characterization of frame representations of graphic matroids, see the work of Chen, DeVos, Funk, and Pivotto~\cite{CDFP2015}.

In the following lemma, we use $K_t$ to denote both the complete graph on $t$ vertices and the $\mathbb{Z}_2$-gain graph whose underlying graph is $K_t$ and all of whose edges are labelled~1. We write $K_{t-1}^{\trivialgroup}$ as $K_{t-1}^1$ for simplicity and regard $K_{t-1}^1$ as the subgraph of $K_{t-1}^{\mathbb{Z}_2}$ obtained by deleting all edges of the form $(-1)_{ij}$. The $\mathbb{Z}_2$-gain graph $H_{t-1}$ is the $\mathbb{Z}_2$-gain graph on $t-1$ vertices such that $E(H_{t-1}) = \{b_1\} \cup \{1_{1i},(-1)_{1i} : 2 \le i \le t-1\} \cup \{1_{ij} : 2 \le i < j \le t-1\}$; we call the unique looped vertex the \emph{center} of $H_{t-1}$. The \emph{doubled triangle} $C_3^{\mathbb{Z}_2}$ is obtained from $K_3^{\mathbb{Z}_2}$ by removing all loops. See Figure~\ref{fig: Z2 Kt} and Figure~\ref{fig:Z2 K4}.

\begin{lemma}\label{lem: Z2 realizations}
    Let $t \ge 3$ and let $G$ be a $\mathbb{Z}_2$-gain graph such that $\FM(G) \cong M(K_t)$. If $t \neq 4$, then $G$ is switching-isomorphic to $K_t$, $K_{t-1}^{1}$, or $H_{t-1}$. If $t = 4$, then $G$ is switching-isomorphic to $K_4$, $K_3^{1}$, $H_3$, or $C_3^{\mathbb{Z}_2}$.
\end{lemma}

\begin{proof}
    Let $G$ be a $\mathbb{Z}_2$-gain graph such that $\FM(G) = M(K_t)$. Let $B$ be a basis of $M(K_t)$ such that every fundamental circuit with respect to $B$ is a triangle, and every pair of elements of $B$ spans a triangle. In $G$, the components of $B$ are trees or unicyclic graphs with an unbalanced cycle. Any unbalanced cycle in $B$ can have at most three edges. If $B$ contains an unbalanced cycle with exactly three edges, then $B$ cannot contain any other elements; in this case $t=4$ and $G$ is switching-isomorphic to $C_3^{\mathbb{Z}_2}$. If $B$ contains a cycle with exactly two edges, then that cycle has a vertex that is incident with every other element of $B$; in this case $G$ is switching-isomorphic to $H_{t-1}$. We may therefore assume that the only cycles $B$ forms in $G$ are loops, in which case $B$ is a $t$-star, a $(t-1)$-star with a loop attached to the center, or $t$ loops. In these cases, $G$ is switching-isomorphic to, respectively, $K_t$, $H_{t-1}$, and $K_{t-1}^{1}$.
\end{proof}

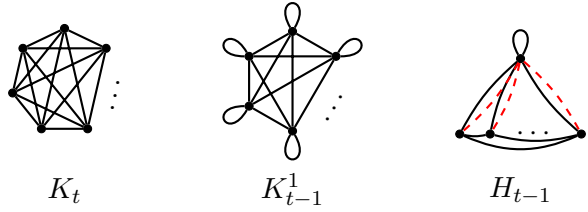
\begin{figure}[htb]
    \centering
    \begin{tikzpicture}
        \begin{scope}[]
            \node[v] (v0) at (90+0*360/7:0.7) {};
            \node[v] (v1) at (90+1*360/7:0.7) {};
            \node[v] (v2) at (90+2*360/7:0.7) {};
            \node[v] (v3) at (90+3*360/7:0.7) {};
            \node[v] (v4) at (90+4*360/7:0.7) {};

            \node (v5) at (90+5*360/7:0.7) {\rotatebox{78}{$\cdots$}};

            \node[v] (vn) at (90+6*360/7:0.7) {};

            \draw[balanced edge]
                (v0)--(v1) (v0)--(v2) (v0)--(v3) (v0)--(v4) (v0)--(vn)
                (v1)--(v2) (v1)--(v3) (v1)--(v4) (v1)--(vn)
                (v2)--(v3) (v2)--(v4) (v2)--(vn)
                (v3)--(v4) (v3)--(vn)
                (v4)--(vn);

            \node () at (0,-1.45) {$K_t$};
        \end{scope}
        
        \begin{scope}[xshift=3cm, yshift=0cm]
            \node[v] (v0) at (90+0*360/6:0.66) {};
            \node[v] (v1) at (90+1*360/6:0.66) {};
            \node[v] (v2) at (90+2*360/6:0.66) {};
            \node[v] (v3) at (90+3*360/6:0.66) {};

            \node (v5) at (90+4*360/6:0.66) {\rotatebox{60}{$\cdots$}};

            \node[v] (vn) at (90+5*360/6:0.66) {};

            \draw[balanced edge]
                (v0)--(v1) (v0)--(v2) (v0)--(v3) (v0)--(vn)
                (v1)--(v2) (v1)--(v3) (v1)--(vn)
                (v2)--(v3) (v2)--(vn)
                (v3)--(vn)
                (v0) to[in=60,out=120,loop] ()
                (v1) to[in=120,out=180,loop] ()
                (v2) to[in=180,out=240,loop] ()
                (v3) to[in=240,out=300,loop] ()
                (vn) to[in=0,out=60,loop] ();

            \node () at (0,-1.45) {$K_{t-1}^1$};
        \end{scope}

        \begin{scope}[xshift=6cm, yshift=-0.7cm]
            \node[v] (v0) at (0,1) {};
            \node[v] (v1) at (-0.8,0) {};
            \node[v] (v2) at (-0.4,0) {};
            \node[v] (vn) at (0.8,0) {};

            \node (v3) at (0.2,0) {$\cdots$};

            \draw[balanced edge]
                (v0) to[in=60,out=120,loop] ()
                (v0)to[bend right=10](v1)
                (v0)to[bend right=10](v2)
                (v0)to[bend right=10](vn)
                (v1)to[bend right=12](v2)
                (v1)to[bend right=24](vn)
                (v2)to[bend right=14](vn);
            \draw[unbalanced edge]
                (v0)to[bend left=10](v1)
                (v0)to[bend left=10](v2)
                (v0)to[bend left=10](vn);

            \node () at (0,-0.75) {$H_{t-1}$};
        \end{scope}
    \end{tikzpicture}
    \caption{The three non-switching-isomorphic $\mathbb{Z}_2$-gain graphs that realize $M(K_t)$ when $t \neq 4$. The solid black edges are labelled~$1$ and the dashed red edges are labelled~$-1$.}
    \label{fig: Z2 Kt}
\end{figure}

\begin{figure}[htb]
    \centering
    \begin{tikzpicture}
        \tikzstyle{v}=[circle, draw, solid, fill=black, inner sep=0pt, minimum width=3pt]
        \begin{scope}[]
            \node[v] (v1) at (-45:0.7) {};
            \node[v] (v2) at (45:0.7) {};
            \node[v] (v3) at (135:0.7) {};
            \node[v] (v4) at (225:0.7) {};

            \draw[balanced edge]
                (v1)--(v2) (v1)--(v3) (v1)--(v4) (v2)--(v3) (v2)--(v4) (v3)--(v4);

            \node () at (0,-1.1) {$K_4$};
        \end{scope}
        
        \begin{scope}[xshift=2.5cm, yshift=-0.2cm]
            \node[v] (v1) at (-30:0.6) {};
            \node[v] (v2) at (90:0.6) {};
            \node[v] (v3) at (210:0.6) {};

            \draw[balanced edge]
                (v1)--(v2)--(v3)--(v1)
                (v1) to[in=-60,out=0,loop] ()
                (v2) to[in=60,out=120,loop] ()
                (v3) to[in=180,out=240,loop] ();

            \node () at (0,-0.9) {$K_3^1$};
        \end{scope}

        \begin{scope}[xshift=5.0cm, yshift=-0.2cm]
            \node[v] (v1) at (-30:0.6) {};
            \node[v] (v2) at (90:0.6) {};
            \node[v] (v3) at (210:0.6) {};

            \draw[balanced edge]
                (v1)to[bend left=25](v2) (v2)to[bend right=25](v3) (v3)--(v1);
            \draw[unbalanced edge]
                (v1)to[bend right=25](v2) (v2)to[bend left=25](v3);
            \draw[balanced edge] (v2) to[in=60,out=120,loop] ();

            \node () at (0,-0.9) {$H_3$};
        \end{scope}

        \begin{scope}[xshift=7.5cm, yshift=-0.2cm]
            \node[v] (v1) at (-30:0.6) {};
            \node[v] (v2) at (90:0.6) {};
            \node[v] (v3) at (210:0.6) {};

            \draw[balanced edge]
                (v1)to[bend right=25](v2) (v2)to[bend right=25](v3) (v3)to[bend right=25](v1);
            \draw[unbalanced edge]
                (v1)to[bend left=25](v2) (v2)to[bend left=25](v3) (v3)to[bend left=25](v1);
            
            \node () at (0,-0.9) {$C_3^2$};
        \end{scope}
    \end{tikzpicture}
    \caption{The four non-switching-isomorphic $\mathbb{Z}_2$-gain graphs realizing $M(K_4)$. The solid black edges are labelled~$1$ and the dashed red edges are labelled~$-1$.}
    \label{fig:Z2 K4}
\end{figure}
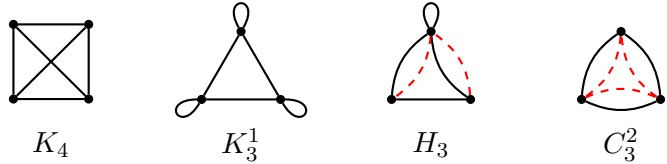

\section{Excluding very long lines and Dowling geometries over nontrivial subgroups}\label{sec: long lines and subgeometries}

In this section we prove Theorems~\ref{thm intro: excluding subgeometry} and~\ref{thm intro: long line}.

\subsection{Joints and very long lines}

Recall that the definition of $Q_n(\Gamma) = \FM(K_n^\Gamma)$ specifies $n$ \emph{joint} elements that correspond to the loops $b_1,\ldots,b_n$ in $K_n^\Gamma$.
We write $J(Q_n(\Gamma))$ for the set of joint elements of $Q_n(\Gamma)$.

We refer to lines with at least four elements as \emph{very long lines}. Very long lines in $Q_n(\Gamma)$ are necessarily spanned by a pair of joint elements. Conversely, when $n \ge 3$ and $\Gamma$ is nontrivial, the joints are the elements that are on at least two very long lines.

The following lemma asserts that if a nontrivial Dowling geometry is embedded in a larger Dowling geometry, then the joints of the smaller geometry must coincide with the joints of the larger geometry. Its proof is immediate from the observation that if $N$ is a submatroid of $M$, and $e$ is a point of $N$ that is contained in two very long lines of $N$, then $e$ must be contained in two very long lines of $M$.

\begin{lemma}\label{lemma:large-subgroup-embedding}
    Let $n \ge t \ge 3$ and let $\Gamma'$ is a nontrivial subgroup of $\Gamma$. If $\psi\colon E(Q_t(\Gamma'))\to E(Q_n(\Gamma))$ is an embedding, then $\psi(J(Q_t(\Gamma'))) \subseteq J(Q_n(\Gamma))$.
    \qed
\end{lemma}

\subsection{Dowling geometries over nontrivial subgroups: Proof of Theorem~\ref{thm intro: excluding subgeometry}}

We prove Theorem~\ref{thm intro: excluding subgeometry} and identify the extremal matroids.

\begin{theorem}\label{thm: excluding subgeometry}
    Let $n \ge t \ge 3$ and let $\Gamma'$ be a nontrivial subgroup of $\Gamma$. Then
    \begin{equation*}
        \ex(Q_n(\Gamma), Q_t(\Gamma')) = |Q_n(\Gamma)| - n + t - 1.
    \end{equation*}
    If $\Gamma'=\Gamma$, the extremal matroids are $Q_n(\Gamma) \setminus X$ where $X$ is an $(n-t+1)$-element set of $Q_n(\Gamma)$ such that (1) for any joint $x\in X$ and non-joint $y\in X$, $x$ is not on the very long line containing $y$; and (2) for any distinct non-joints $x,y\in X$, the two very long lines containing $x$ and $y$ do not intersect.
    If $\Gamma'$ is a proper nontrivial subgroup of $\Gamma$, then the extremal matroids are $Q_n(\Gamma) \setminus X$ where $X$ is an $(n-t+1)$-element subset of $J(Q_n(\Gamma))$.
\end{theorem}

\begin{proof}
    We first show that the given matroids are $Q_t(\Gamma')$-free, and hence that $\ex(Q_n(\Gamma), Q_t(\Gamma')) \ge |Q_n(\Gamma)| - n + t - 1$.
    Let $X$ be an $(n-t+1)$-element subset of $J(Q_n(\Gamma))$; by Lemma~\ref{lemma:large-subgroup-embedding}, the matroid $Q_n(\Gamma)\backslash X$ is $Q_t(\Gamma')$-free.
    Now let $\Gamma' = \Gamma$ and let $X \subseteq E(Q_n(\Gamma))$ be an $(n-t+1)$-element set satisfying (1) and~(2). We show that $Q_n(\Gamma)\backslash X$ is $Q_t(\Gamma)$-free.
    Suppose that $X = \{e_1,\ldots,e_k,f_{k+1},\ldots,f_{n-t+1}\}$ where $e_i \in J(Q_n(\Gamma))$ for $i \in \{1,\ldots,k\}$ and $f_j \not\in J(Q_n(\Gamma))$ for $j \in \{k+1,\ldots,n-t+1\}$.
    For $j \in \{k+1,\ldots,n-t+1\}$, let $e_j$ and $e'_j$ be the two joints on the unique very long line containing $f_j$, and let $X' = \{e_1, \ldots, e_k, e_{k+1}, e_{k+1}', \ldots,e_{n-t+1}, e_{n-t+1}'\}$. By~(1) and~(2), all elements of $X'$ are distinct.
    Note that $X'$ contains $n-t-k+1$ elements of the form $e'_j$ for $j\in\{k+1,\ldots, n-t+1\}$ to give a total of $2n-2t-k+2$ elements.
    Let $T \subseteq E(Q_n(\Gamma))$ be such that $Q_n(\Gamma)|T \cong Q_t(\Gamma')$. By Lemma~\ref{lemma:large-subgroup-embedding}, the joints of $Q_n(\Gamma)|T$ are joints of $Q_n(\Gamma)$, so as $|J(Q_n(\Gamma)|T)|+|X'|=2n-t-k+2$ and $|J(Q_n(\Gamma))|=n$, we have $|J(Q_n(\Gamma)|T)\cap X'|\geq n-t-k+2$.
    By the pigeonhole principle, $J(Q_n(\Gamma)|T)$ contains an element $e_i$ with $i \in \{1,\ldots,k\}$ or a pair $\{e_j,e'_j\}$ with $j \in \{k+1,\ldots,n-t+1\}$. In the former case, $T$ contains $e_i$, and in the latter case, $T$ contains $f_j$. In either case, $T$ contains an element of $X$, so $Q_n(\Gamma)|T$ is not contained in $Q_n(\Gamma)\backslash X$.

    In order to show that no larger $Q_t(\Gamma')$-free restriction of $Q_n(\Gamma)$ exists, let $X$ be an $(n-t)$-element subset of $E(Q_n(\Gamma))$, say $X = \{e_1, \ldots, e_k, f_{k+1}, \ldots, f_{n-t}\}$ where $e_i \in J(Q_n(\Gamma))$ for $i \in \{1,\ldots,k\}$ and $f_j \not\in J(Q_n(\Gamma))$ for $j\in\{k+1,\ldots,n-t\}$.
    For $j \in \{k+1, \ldots, n-t\}$, the element $f_j$ is in a unique very long line of $Q_n(\Gamma)$; let $e_j$ be one of the two joints on this line (we allow $e_i = e_j$ even when $i \neq j$).
    Let $S$ be a $t$-element subset of $J(Q_n(\Gamma)) \setminus \{e_1, \ldots, e_{n-t}\}$. Note that $\cl(S)$ contains only very long lines that are spanned by pairs of elements of $S$, so $\cl(S) \cap X = \emptyset$. Thus, 
    $Q_n(\Gamma)\backslash X\supseteq Q_n(\Gamma)|\cl(S)\cong Q_t(\Gamma)\supseteq Q_t(\Gamma')$.
    This gives us the upper bound $\ex(Q_n(\Gamma), Q_t(\Gamma')) \leq |Q_n(\Gamma)| - n + t - 1$.

    Finally, we show that the extremal matroids are as stated. Let $X$ be an $(n-t+1)$-element subset of $E(Q_n(\Gamma))$ such that $Q_n(\Gamma)\setminus X$ is $Q_t(\Gamma')$-free. As before, suppose that $X = \{e_1,\ldots,e_k, f_{k+1}, \ldots, f_{n-t+1}\}$ where $e_i \in J(Q_n(\Gamma))$ for $i\in\{1,\ldots,k\}$ and $f_j \not\in J(Q_n(\Gamma))$ for  $j\in\{k+1,\ldots,n-t+1\}$. Choose $\{e_{k+1}, \ldots, e_{n-t+1}\}$, where $e_j$ is one of the joints on the very long line containing $f_j$, for all $j=k+1,\ldots,n-t+1$, such that $|\{e_1, ..., e_{n-t+1}\}|$ is as small as possible among all such sets.

    If $\Gamma' = \Gamma$ and (1) and~(2) do not both hold, then $e_i = e_j$ for some $i \neq j$, so there exists a $t$-element subset $S$ of $J(Q_n(\Gamma)) \setminus \{e_1, \ldots, e_{n-t+1}\}$. Then $Q_n(\Gamma)\setminus X \supseteq Q_n(\Gamma)|\cl(S) \cong Q_t(\Gamma)$, so $Q_n(\Gamma)\setminus X$ has a restriction isomorphic to $Q_t(\Gamma')$.

    If $\Gamma' < \Gamma$, we may assume that $X$ contains a non-joint element (so $k \le n-t$). If $n=t$, then $X = \{f_1\}$. By switching so that $f_1$ corresponds to an edge with gain in $\Gamma\setminus\Gamma'$, we have that $Q_t(\Gamma)\setminus f_1$ contains $Q_t(\Gamma')$ as a submatroid. So we may assume that $n > t$. Let $S$ be a $t$-element subset of $J(Q_n(\Gamma)) \setminus \{e_1, \ldots, e_{n-t}\}$. Then $Q_n(\Gamma)\setminus (X\setminus \{f_{n-t+1}\}) \supseteq Q_n(\Gamma)|\cl(S) \cong Q_t(\Gamma)$. Since $f_{n-t+1}$ is not a joint of $Q_n(\Gamma)|\cl(S)$, it follows that $Q_n(\Gamma)\setminus X$ contains $Q_t(\Gamma')$ as a submatroid.
\end{proof}

\subsection{Very long lines: Proof of Theorem~\ref{thm intro: long line}}

We prove Theorem~\ref{thm intro: long line} and identify the extremal matroids.

\begin{theorem}\label{thm: long line}
    Let $n \ge 2$ and $\ell \ge 4$.

    \begin{enumerate}[(i)]
        \item
        If $\ell \ge |\Gamma|+3$, then $\ex(Q_n(\Gamma),U_{2,\ell}) = |Q_n(\Gamma)|$.

        \item 
        If $\ell = |\Gamma|+2$, then $\ex(Q_n(\Gamma), U_{2,\ell}) = |Q_n(\Gamma)|-n+1$; the extremal matroids are 
        \begin{itemize}
            \item
            $Q_n(\Gamma)\backslash X$, where $X$ is an $(n-1)$-subset of $J(Q_n(\Gamma))$, and

            \item 
            $Q_n(\Gamma)\backslash (X \cup \{f\})$, where $X$ is an $(n-2)$-subset of $J(Q_n(\Gamma))$ and $f \in \cl(J(Q_n(\Gamma))\setminus X)\setminus J(Q_n(\Gamma))$.
        \end{itemize}

        \item 
        If $\ell \le |\Gamma|+1$, then $\ex(Q_n(\Gamma), U_{2,\ell}) = (\ell-1)\binom{n}{2}$; the extremal matroids are $Q_n(\Gamma)|X$, where $X \subseteq E(Q_n(\Gamma)) \setminus J(Q_n(\Gamma))$ contains exactly $\ell-1$ elements from each of the very long lines of $Q_n(\Gamma)$.
    \end{enumerate}
\end{theorem}

\begin{proof}
    The first statement is immediate, as lines in $Q_n(\Gamma)$ have length at most $|\Gamma|+2$.

    Let $M \subseteq Q_n(\Gamma)$ be a $U_{2,\ell}$-free submatroid with the maximum number of elements. Let $K = E(M) \cap J(Q_n(\Gamma))$ and let $k = |K|$.

    If $\ell = |\Gamma|+2$, then, for any distinct joints $x, y \in J(Q_n(\Gamma))$,
    \[
        |( \cl_{Q_n(\Gamma)}(\{x,y\}) \setminus \{x,y\} ) \cap E(M)|
        =
        \begin{cases}
            |\Gamma|-1 & \text{if $x,y\in K$}, \\
            |\Gamma| & \text{otherwise}.
        \end{cases}
    \]
    As $\ell \ge 4$, it follows that 
    \begin{equation*}
        |M| 
        =
        k + \binom{k}{2}(|\Gamma|-1) + k(n-k)|\Gamma| + \binom{n-k}{2}|\Gamma| = |\Gamma|\binom{n}{2} + k - \binom{k}{2}.
    \end{equation*}
    The right-hand side is maximised for $k = 1$ and $k=2$, in which case $k-\binom{k}{2} = 1$. When $k=1$, $M$ must be an extremal matroid of the first type, and when $k=2$, $M$ must be an extremal matroid of the second type.

    Finally, assume that $4 \le \ell \le |\Gamma|+1$. As $\ell \ge 4$, it follows that 
    \begin{equation*}
        |M| = k + \binom{k}{2}(\ell-3) + k(n-k)(\ell-2) + \binom{n-k}{2}(\ell-1) = (\ell-1)\binom{n}{2} - k(n-2),
    \end{equation*}
    and so $\ex(Q_n(\Gamma),U_{2,\ell}) \le (\ell-1)\binom{n}{2}$. If equality holds, we must have $k=0$, and $M$ must contain exactly $\ell-1$ elements from each of the $\binom{n}{2}$ very long lines of $Q_n(\Gamma)$.
\end{proof}

\section{Mantel's theorem for Dowling geometries}\label{sec: Mantel}

In this section, we prove Theorem~\ref{thm intro: Mantel for DG}, which is obtained by combining Theorems~\ref{thm: Mantel}--\ref{thm: Mantel Z3} and~\ref{thm: Mantel Gamma} below.

When the group $\Gamma$ is trivial, $Q_n(\Gamma) \cong M(K_{n+1})$ and the Tur\'{a}n problem for triangles is just Mantel's theorem~\cite{Mantel1910} in disguise.
\begin{theorem}\label{thm: Mantel}
    $\ex(Q_{n-1}(\trivialgroup), M(K_3)) = \lfloor n^2/4\rfloor$. The extremal matroids are isomorphic to the cycle matroid of a balanced complete bipartite graph $K_{\lceil n/2 \rceil, \lfloor n/2 \rfloor}$. \qed
\end{theorem}

Next, let $\Gamma = \mathbb{Z}_2$, the multiplicative group on $\{1,-1\}$. Recall that any $\mathbb{Z}_2$-gain graph that realizes $M(K_3)$ is switching-isomorphic to $K_3$, $K_2^{1}$ or $H_2$.

Let $K_{a,b}^{\mathbb{Z}_2}$ be the $\mathbb{Z}_2$-gain graph on vertices $[a+b]$, with the $2ab$ edges $(\pm 1)_{ij}$, where $1 \le i \le a$ and $a+1 \le j \le a+b$.

\begin{theorem}\label{thm: Mantel Z2}
    $\ex(Q_n(\mathbb{Z}_2), M(K_3)) = \lfloor n^2/2\rfloor$ for $n \ge 2$. If $G \subseteq K_n^{\mathbb{Z}_2}$ is a $\mathbb{Z}_2$-gain graph that induces an extremal matroid, then $G$ is isomorphic to $K_{\lceil n/2\rceil, \lfloor n/2\rfloor}^{\mathbb{Z}_2}$. In particular, any extremal matroid is isomorphic to $\FM(K_{\lceil n/2\rceil, \lfloor n/2\rfloor}^{\mathbb{Z}_2})$.
\end{theorem}

\begin{proof}
    The statement is immediate for $n=2$, as $Q_2(\mathbb{Z}_2) \cong U_{2,4}$. For $n=3$, any extremal matroid is isomorphic to $U_{3,4} \cong \FM(G_{1,2})$; it is worth noting that up to switching-isomorphism, there are four extremal gain subgraphs of $K_3^{\mathbb{Z}_2}$, as depicted in Figure~\ref{fig: three descriptions of U34}. We may therefore assume that $n \ge 4$.

    We begin by establishing the lower bound. The gain graph $K_{a,b}^{\mathbb{Z}_2}$ does not have any subgraph switching-isomorphic to $K_3$, $K_2^1$, or $H_2$. Setting $a = \lceil n/2\rceil$ and $b = \lfloor n/2\rfloor$ verifies that $\ex(Q_n(\mathbb{Z}_2), M(K_3)) \ge 2\lceil n/2\rceil\lfloor n/2\rfloor=\lfloor n^2/2\rfloor$.

    We prove the upper bound by induction. Let $M \subseteq Q_n(\mathbb{Z}_2)$ be a triangle-free matroid, and let $G$ be an $n$-vertex $\mathbb{Z}_2$-gain subgraph of $K_n^{\mathbb{Z}_2}$ such that $\FM(G) = M$, and let $k$ be the number of loops in $G$.

    If, for all pairs of vertices $i < j$, the gain graph $G$ contains at most one of the edges $1_{ij}$ and $(-1)_{ij}$, then
    \begin{equation*}
        |M| = |E(G)| \le k + \binom{n}{2} - \binom{k}{2} \le \binom{n}{2} + 1 < \lfloor n^2/2\rfloor,
    \end{equation*}
    where the last inequality holds since $n \ge 4$.
    So we may assume $G$ has a pair of vertices $i < j$ such that both $1_{ij}$ and $(-1)_{ij}$ are edges.

    Let $G'$ be obtained from $G$ by removing the vertices $i$ and $j$ and all edges incident with them. As $G$ does not contain a subgraph switching-isomorphic to $K_3$, $K_2^1$, or $H_2$, neither does $G'$, so by induction $|E(G')| \le 2\lfloor (n-2)/2\rfloor \lceil (n-2)/2\rceil$ with equality if and only if $G'$ is isomorphic to $K^{\mathbb{Z}_2}_{\lfloor (n-2)/2\rfloor,\lceil (n-2)/2\rceil}$.

    For each $v \neq i,j$, the gain graph $G$ contains at most two edges from $\{1_{iv},(-1)_{iv},1_{jv},(-1)_{jv}\}$, for if $G$ contains at least three edges from this set, these edges together with the edges $\{1_{ij},(-1)_{ij}\}$ contains a subgraph switching-isomorphic to $K_3$.

    Combining these estimates, we find that
    \begin{equation}\label{eq:mantel Z2 upper bound}
        |M| = |E(G)| \le 2 + 2(n-2) + 2\lfloor (n-2)/2\rfloor \lceil (n-2)/2\rceil = \lfloor n^2/2\rfloor,
    \end{equation}
    which gives us the extremal function.

    Equality holds in~\eqref{eq:mantel Z2 upper bound} if $G'$ is isomorphic to $K^{\mathbb{Z}_2}_{\lfloor (n-2)/2\rfloor,\lceil (n-2)/2\rceil}$ and for each $v \not\in \{i,j\}$, $G$ contains exactly two of the edges in $\{1_{iv},(-1)_{iv},1_{jv},(-1)_{jv}\}$. When the latter happens, these two edges are necessarily $\{1_{iv},(-1)_{iv}\}$ or $\{1_{jv},(-1)_{jv}\}$, as otherwise $G$ would contain a subgraph switching-isomorphic to $K_3$.
    Let $U_i = \{v \in V(G') : 1_{iv} \in E(G)\}$ and define $U_j$ similarly. Then $U_i$ and $U_j$ partition $V(G')$. As $G$ cannot contain an edge both of whose end points are in $U_i$ or both of whose end points are in $U_j$, it follows that $\{|U_i|, |U_j|\} = \{\lfloor (n-2)/2 \rfloor, \lceil (n-2)/2 \rceil\}$ and that $G$ contains every edge with one endpoint in $U_i$ and the other end point in $U_j$. Now $U_i\cup\{j\}$ and $U_j\cup\{i\}$ gives the required bipartition of $V(G)$.
\end{proof}

\begin{figure}[t]
    \centering
    \begin{tikzpicture}
        \begin{scope}[xshift=-1.5cm, yshift=0cm]
            \node[v] (v1) at (-30:0.6) {};
            \node[v] (v2) at (90:0.6) {};
            \node[v] (v3) at (210:0.6) {};

            \draw[balanced edge]
                (v1)to[bend right=25](v2) (v2)to[bend right=25](v3) (v3)to[bend right=25](v1);
            \draw[unbalanced edge]
                (v1)to[bend left=25](v2) (v2)to[bend left=25](v3) (v3)to[bend left=25](v1);

            \draw[balanced edge] (v1) to[in=-60,out=0,loop] ();
            \draw[balanced edge] (v2) to[in=60,out=120,loop] ();
            \draw[balanced edge] (v3) to[in=180,out=240,loop] ();
            
            \node () at (0,-1.0) {$K_3^{\mathbb{Z}_2}$};
        \end{scope}

        \begin{scope}[xshift=2.7cm, yshift=0cm]
            \node[v] (v1) at (-30:0.6) {};
            \node[v] (v2) at (90:0.6) {};
            \node[v] (v3) at (210:0.6) {};

            \draw[balanced edge]
                (v1)to[bend right=25](v2)
                (v2)to[bend right=25](v3) ;
            \draw[unbalanced edge]
                (v1)to[bend left=25](v2)
                (v2)to[bend left=25](v3) ;

        \end{scope}

        \begin{scope}[xshift=5.4cm, yshift=0cm]
            \node[v] (v1) at (-30:0.6) {};
            \node[v] (v2) at (90:0.6) {};
            \node[v] (v3) at (210:0.6) {};

            \draw[balanced edge]
                (v1)to[bend right=25](v2)
                (v2)to[bend right=25](v3);
            \draw[unbalanced edge]
                (v1)to[bend left=25](v2);

            \draw[balanced edge] (v3) to[in=180,out=240,loop] ();
            
        \end{scope}

        \begin{scope}[xshift=8.1cm, yshift=0cm]
            \node[v] (v1) at (-30:0.6) {};
            \node[v] (v2) at (90:0.6) {};
            \node[v] (v3) at (210:0.6) {};

            \draw[balanced edge]
                (v1)to[bend right=25](v2)
                (v2)to[bend right=25](v3);
            \draw[unbalanced edge]
                (v3)to[bend right=25](v1);

            \draw[balanced edge] (v3) to[in=180,out=240,loop] ();
            
        \end{scope}

        \begin{scope}[xshift=10.8cm, yshift=0cm]
            \node[v] (v1) at (-30:0.6) {};
            \node[v] (v2) at (90:0.6) {};
            \node[v] (v3) at (210:0.6) {};

            \draw[balanced edge]
                (v1)to[bend right=25](v2)
                (v2)to[bend right=25](v3);

            \draw[balanced edge] (v1) to[in=-60,out=0,loop] ();
            
            \draw[balanced edge] (v3) to[in=180,out=240,loop] ();
            
        \end{scope}
    \end{tikzpicture}
    \caption{$K_3^{\mathbb{Z}_2}$ (leftmost) and its subgraphs realizing $U_{3,4}$ (right four illustrations).}
    \label{fig: three descriptions of U34}
\end{figure}
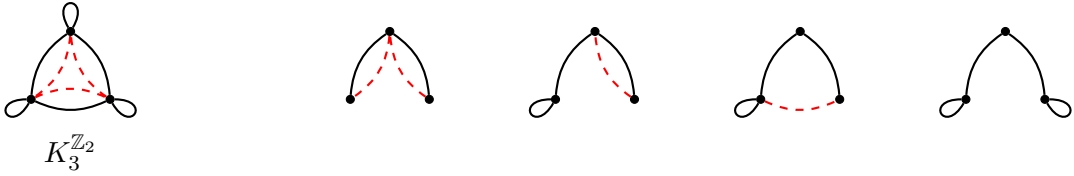

We now move to the case $\Gamma = \mathbb{Z}_3$. Let $x$ be an element of $\mathbb{Z}_3$ of order $3$, so $\mathbb{Z}_3 = \{1,x,x^{-1}\}$.

\begin{theorem}\label{thm: Mantel Z3}
    $\ex(Q_n(\mathbb{Z}_3), M(K_3)) = \lceil n^2/2\rceil$ for all $n \ge 1$. If $M \subseteq Q_n(\mathbb{Z}_3)$ is an extremal matroid, and $G$ is an $n$-vertex $\mathbb{Z}_3$-gain subgraph of $K_n^{\mathbb{Z}_3}$, then $G$ has at most two loops.
\end{theorem}

\begin{proof}
    The statement is immediate for $n\le 2$, and follows by a straightforward case analysis for $n=3$. We may therefore assume that $n \ge 4$.

    We first show the lower bound. If $n$ is even, let $G$ be the subgraph of $K_n^{\mathbb{Z}_3}$ obtained by keeping only the edges of the form $x_{ij}$ and $x^{-1}_{ij}$ with $i \le n/2$ and $j > n/2$. If $n$ is odd, let $G$ be the subgraph of $K_n^{\mathbb{Z}_3}$ obtained by keeping only the edges of the form $x_{ij}$ and $x^{-1}_{ij}$ for $i < n/2$ and $n/2 < j < n$, as well as the loop at vertex $n$ and all edges of the form $1_{in}$ for $i < n$. In both cases, it is straightforward to verify that $G$ has $\lceil n^2/2\rceil$ edges, and no subgraph switching-isomorphic to a 3-edge theta, 3-edge handcuff, or $K_3$.

    We prove the upper bound by induction. Let $M \subseteq Q_n(\mathbb{Z}_3)$ be an extremal matroid, and let $G$ be an $n$-vertex $\mathbb{Z}_3$-gain subgraph of $K_n^{\mathbb{Z}_3}$ such that $\FM(G) = M$. If $G$ does not contain any parallel edges and has exactly $k$ loops, then
    \begin{equation*}
        |E(G)| \le k + \binom{n}{2} - \binom{k}{2} \le \binom{n}{2} + 1 < \lceil n^2/2\rceil,
    \end{equation*}
    so we may assume that $G$ has a pair of vertices $i < j$ such that $G$ contains exactly two edges from $\{1_{ij}, x_{ij}, x^{-1}_{ij}\}$. Let $G'$ be the $\mathbb{Z}_3$-gain graph obtained from $G$ by removing the vertices $i$ and $j$ together with all edges incident with those vertices. Then $G'$ is a subgraph of $K_{n-2}^{\mathbb{Z}_3}$ that does not contain a subgraph switching-isomorphic to a 3-edge theta, a 3-edge handcuff, or $K_3$, so by induction $G'$ has at most $\lceil (n-2)^2/2\rceil$ edges. For any vertex $v \not\in \{i,j\}$, if $G$ contains three edges of the form $\alpha_{uv}$ with $\alpha \in \mathbb{Z}_3$ and $u \in \{i,j\}$, then $G$ contains a subgraph switching-isomorphic to a 3-edge theta or a balanced $K_3$, so $G$ contains at most two of such edges. It follows that
    \begin{equation*}
        |E(G)| \le 2 + 2(n-2) + \lceil (n-2)^2/2\rceil = \lceil n^2 /2\rceil.
    \end{equation*}

    Finally, we prove that $G$ has at most two loops. If $G$ has $\lceil n^2/2\rceil$ edges of which $k$ are loops, then, by an argument similar as before,
    \begin{equation*}
        \lceil n^2/2\rceil = |E(G)| \le k + k(n-k) + \lceil (n-k)^2/2\rceil,
    \end{equation*}
    which fails if $k \ge 3$.
\end{proof}

We remark that there is a rich family of extremal $\mathbb{Z}_3$-gain graphs. For example, when $n$ is even, and $U_1 \cup V_1 \cup U_2 \cup V_2 \cup \cdots \cup U_k \cup V_k$ is a partition of the vertices of $K_n^{\mathbb{Z}_3}$ such that $|U_s| = |V_s|$ for all $s \in [k]$, the subgraph of $K_n^{\mathbb{Z}_3}$ comprising the edges of the form $x_{ij}$ and $x^{-1}_{ij}$ for $(i,j) \in U_s \times V_s$ and all edges $1_{ij}$ for $(i,j) \in (U_s \cup V_s) \times (U_t \cup V_t)$ for $s \neq t$ is an extremal graph.

So far, we have only considered finite groups of order at most~$3$. We now focus on larger groups.
A pair of elements $\{x, y\} \subseteq \Gamma$ is called a \emph{good pair} if $x$ and $y$ are distinct, neither is the identity in $\Gamma$, and moreover $x^2 \neq y$ and $y^2 \neq x$.

\begin{lemma}
    A group $\Gamma$ has a good pair if and only if $\Gamma$ has order at least~$4$.
\end{lemma}
\begin{proof}
    If the group $\Gamma$ has order less than four, then $\Gamma \cong \mathbb{Z}_k$ for $k \in \{1,2,3\}$. It is easily checked that $\Gamma$ does not have a good pair in any of these cases. So we may assume that $\Gamma$ has order at least~4.

    Suppose that $\Gamma$ has an element $g$ of order at least~$4$; then the pair $\{g, g^{-1}\}$ is good. Otherwise, let $x$ be an element of order~2 or~3, and let $y \in \Gamma\setminus\langle x\rangle$; again $\{x, y\}$ is a good pair.
\end{proof}

\begin{theorem}\label{thm: Mantel Gamma}
    Let $\Gamma$ be a finite group of order at least~4. Then $\ex(Q_n(\Gamma), M(K_3)) = 2\binom{n}{2} = n(n-1)$ for all $n\ge 2$. If $n \ge 3$, then any triangle-free matroid $M \subseteq Q_n(\Gamma)$ on $2\binom{n}{2}$ elements does not contain any joint of $Q_n(\Gamma)$.
\end{theorem}

\begin{proof}
    Let $\{x, y\}$ be a good pair in $\Gamma$. Consider the $\Gamma$-gain graph $G$ on vertex set $[n]$ with the $2\binom{n}{2}$ edges of the form $\alpha_{ij}$ where $\alpha \in \{x,y\}$ and $1 \le i < j \le n$. By construction, $G$ has no three-element handcuffs or thetas. Since the pair $\{x,y\}$ is good, $G$ does not have any balanced triangles. It follows that $\FM(G)$ is a triangle-free submatroid of $Q_n(\Gamma)$ on $2\binom{n}{2}$ elements.

    Now let $n \ge 3$ and let $G \subseteq K_n^\Gamma$ be an $n$-vertex $\Gamma$-gain graph that does not contain any three-element theta or handcuff. Then for each pair of vertices $i < j$, $G$ contains at most two of the edges in the set $\{\alpha_{ij} : \alpha \in \Gamma\} \cup \{b_i, b_j\}$. As these $\binom{n}{2}$ sets cover $E(K_n^\Gamma)$, it follows that $G$ contains at most $2\binom{n}{2}$ edges.
    If equality holds, then each such set must contribute exactly two elements, so $G$ cannot contain any of the elements $b_i$.
\end{proof}

\section{Tur\'{a}n density of graphic matroids}\label{sec: graphic}

\subsection{General upper bound: Proof of Theorem~\ref{thm intro: gain ESS}}

The following theorem proves a more general version of Theorem~\ref{thm intro: gain ESS}.
\begin{theorem}\label{thm:gain ESS}
    Let $\Gamma' \leq \Gamma$ be finite groups and let $N$ be a matroid.
    Then 
    \begin{equation*}
        |\Gamma|^{-1} \ex(Q_n(\Gamma),N) \le |\Gamma'|^{-1} \ex(Q_n(\Gamma'),N).
    \end{equation*}
    In particular,
    \begin{equation*}
        \pi(\Gamma, M(H)) \le \frac{\chi(H)-2}{\chi(H)-1}
        \qquad\text{for all graphs $H$.}
    \end{equation*}
\end{theorem}

\begin{proof}
    Let $G$ be a subgraph of $K_n^{\Gamma}$ such that $|E(G)| > \frac{|\Gamma|}{|\Gamma'|} \ex(Q_n(\Gamma'),N)$.
    We claim that there exists a graph $G'$ that is switching-equivalent to a subgraph of $G$ such that $|E(G')| > \ex(Q_n(\Gamma'),N)$ and all non-loop edges of $G'$ are labelled elements of $\Gamma'$.
    Indeed, such a subgraph can be obtained by the following procedure.
    Starting from $G$, for $j = 2, 3, \ldots, n$, in that order, switch at vertex $j$ of $G$ so that  at least a $|\Gamma'|/|\Gamma|$-fraction of edges $x_{ij}$ with $i<j$ is labelled an element from $\Gamma'$. Finally, delete all the non-loop edges that are not labelled by $\Gamma'$.
    Then $N$ is a submatroid of $\FM(G')$ and thus it is a submatroid of $\FM(G)$.
    Therefore, $\ex(Q_n(\Gamma),N) \le \frac{|\Gamma|}{|\Gamma'|} \ex(Q_n(\Gamma'),N)$.
    
    The second statement follows immediately upon combining the first statement with $\Gamma'$ the trivial group and the Erd\H{o}s--Stone--Simonovits theorem.
\end{proof}

The Erd\H{o}s--Stone--Simonovits theorem implies that the asymptotic behaviour of $\ex(n,H)$ is quadratic unless $H$ is a bipartite graph. This observation extends to extremal numbers for Dowling geometries.

\begin{corollary}\label{cor: bipartite graphic matroid}
    Let $\Gamma$ be a finite group and let $N$ be a matroid.
    Then $\pi(\Gamma, N) = 0$ if and only if $N$ is a bipartite graphic matroid, i.e., $N = M(H)$ for some bipartite graph $H$.
\end{corollary}
\begin{proof}
    The backward direction is immediate from Theorem~\ref{thm:gain ESS}.
    To prove the forward direction, suppose that $N$ is not the cycle matroid of any bipartite graph.
    Then $N$ is non-graphic or $N$ has an odd circuit, so $M(K_{\lfloor n/2 \rfloor, \lceil n/2 \rceil})$ is $N$-free.
    It follows that $\ex(Q_n(\Gamma),N) \ge \lfloor n^2/2 \rfloor$ and so $\pi(\Gamma,N) \ge \frac{1}{2|\Gamma|}$.
\end{proof}

In contrast to the Erd\H{o}s--Stone--Simonovits theorem, Theorem~\ref{thm:gain ESS} only bounds $\pi(\Gamma,M(H))$ from above. We now introduce a criterion that is sufficient for equality.
\begin{definition}
	A pair $(\Gamma, H)$ of a finite group $\Gamma$ and a graph $H$ is called a \emph{critical pair} (we also say that $H$ is \emph{$\Gamma$-critical}) if there does not exist a loopless, nonempty $\Gamma$-labelled graph $H'$ such that $\FM(H) \cong \FM(H')$ such that $\chi(\underline{H}') < \chi(H)$.
\end{definition}

Let $T_{n,k}^\Gamma$ be the subgraph of $K_n^\Gamma$ obtained by replacing each edge of the Tur\'{a}n graph $T_{n,k}$ by $|\Gamma|$ labelled edges. More precisely, $T_{n,k}^\Gamma$ is obtained by partitioning the vertex set of $K_n^\Gamma$ into $k$ classes, all of which have size $\lfloor n/k\rfloor$ or $\lceil n/k\rceil$ and by keeping only the edges that connect vertices between different classes.

\begin{proposition}
    Let $\Gamma$ be a finite group, and let $H$ be a simple graph with at least one edge. If $H$ is $\Gamma$-critical, then $\pi(\Gamma, M(H)) = \frac{\chi(H)-2}{\chi(H)-1}$.
\end{proposition}

\begin{proof}
    When $H$ is bipartite, this follows from Corollary~\ref{cor: bipartite graphic matroid}, so we may assume that $\chi(H) \ge 3$. Since $H$ is $\Gamma$-critical, $T_{n,\chi(H)-1}^\Gamma$ has no subgraph isomorphic to $H'$ such that $\FM(H') \cong \FM(H)$. Therefore $\ex(Q_n(\Gamma),M(H)) \ge |E(T_{n,\chi(H)-1}^\Gamma)| = \frac{\chi(H)-2}{\chi(H)-1} |\Gamma| \binom{n}{2} + o(n^2)$. The claim now follows from Theorem~\ref{thm:gain ESS}.
\end{proof}

It is immediate from the definition that a bipartite graph $H$ with at least one edge is $\Gamma$-critical for every group $\Gamma$. It follows from Lemma~\ref{lem: Z2 realizations} that the same conclusion holds for a clique with at least five vertices. The triangle $K_3$ is $\Gamma$-critical if and only if $|\Gamma| \le 2$: for larger groups, $M(K_3)$ is realized by a 2-vertex theta-graph. The 4-clique $K_4$ is $\Gamma$-critical if and only if $\Gamma$ does not have an element of order~2: if $x$ is an element of order 2, then $M(K_4)$ is realized by a doubling of $K_3$ whose edges are labelled $1$ and $x$.
For $t \ge 2$, the odd cycle $C_{2t+1}$ is not $\Gamma$-critical for any nontrivial group $\Gamma$: it is realized by a loose handcuff consisting of two pairs of parallel edges connected by a $(2t-3)$-edge path.

When $(\Gamma, H)$ is not a critical pair, it may still be true that $\pi(\Gamma, H) = \frac{\chi(H)-2}{\chi(H)-1}$ -- for example when $\Gamma$ is a group of order~4 and $H=K_3$ -- or $\pi(\Gamma, H) < \frac{\chi(H)-2}{\chi(H)-1}$, for example when $\Gamma$ is a group of order at least~5 and $H=K_3$.

The next lemma shows that equality holds for the non-critical pair $(\mathbb{Z}_2, K_4)$. In Theorem~\ref{thm: Z2 K4} we obtain the exact extremal number in this case.
\begin{lemma}\label{lem: K4 lower}
    Let $\Gamma$ be a finite group of order $m\ge 2$.
    Then $\pi(\Gamma, M(K_4)) \ge \frac{m}{2m-1}$.
\end{lemma}
\begin{proof}
    Let $x\in \Gamma$ be a non-identity element and let $n' := \lfloor \frac{m}{2m-1} n \rfloor$.
    Let $G$ be the subgraph of $K_n^{\Gamma}$ whose edges are
    \begin{equation*}
        \{x_{ij} : 1 \le i < j \le n'\} \cup
        \{y_{ij} : y \in \Gamma\text{ and }1 \le i \le n' < j \le n\}.
    \end{equation*}
    Then $G$ has $\binom{n'}{2} + m n' (n-n') = \frac{m^2}{4m-2} n^2 + O(mn)$ edges, and $\FM(G)$ is $M(K_4)$-free.
\end{proof}

\subsection{The non-critical pairs \texorpdfstring{$(\mathbb{Z}_4, K_4)$}{(Z4,K4)} and \texorpdfstring{$(\mathbb{Z}_2 \times \mathbb{Z}_2, K_4)$}{(Z2xZ2,K4)} : Proof of Theorem~\ref{thm intro: order 4}}

In this section, we consider the non-critical pairs $(\mathbb{Z}_4, K_4)$ and $(\mathbb{Z}_2 \times \mathbb{Z}_2, K_4)$ and show that
\begin{equation*}
    \pi(\mathbb{Z}_2 \times \mathbb{Z}_2, M(K_4)) < \pi(\mathbb{Z}_4, M(K_4)) \le \frac{2}{3}.
\end{equation*}

\begin{theorem}
    $\frac{8}{13} \le \pi(\mathbb{Z}_4, M(K_4)) \le \frac{2}{3}$.
\end{theorem}

\begin{proof}
    The upper bound is a direct consequence of Theorem~\ref{thm:gain ESS}.
    
    To prove the lower bound, we construct a dense $\mathbb{Z}_4$-gain subgraph $G$ of $K_n^{\mathbb{Z}_4}$ such that $\FM(G)$ is $M(K_4)$-free.
    Let $m:= \lfloor \frac{4}{13} n \rfloor$.
    Let $A = [m]$, $B = [2m]\setminus [m]$, and $C = [n]\setminus [2m]$. Let $\alpha$ be a generator of $\mathbb{Z}_4$, so $\mathbb{Z}_4 = \{1,\alpha,\alpha^2,\alpha^3\}$.
    Let $G$ be a subgraph of $K_n^{\mathbb{Z}_4}$ such that its edges are
    \begin{equation*}
        \{\alpha^2_{ij} : i,j \in A\} \cup \{\alpha^2_{ij} : i,j \in B\} \cup \{1_{ij},\alpha_{ij} : i \in A, j \in B\} \cup \{1_{ij}, \alpha_{ij}, \alpha^2_{ij}, \alpha^3_{ij} : i \in A\cup B, j \in C\}.
    \end{equation*}
    Then $G$ has $\frac{16}{13}n^2 + O(n)$ edges.
    It remains to check that $\FM(G)$ does not contain $M(K_4)$ as a submatroid.
    Suppose, contrary to our claim, that $G$ has a subgraph $H$ such that $\FM(H) \cong M(K_4)$.
    In $G[A\cup B]$, there is no balanced triangle.
    Hence, $H$ consists of three vertices, say $u,v,w$, such that $u,v\in A\cup B$ and $w\in C$.
    We may assume that $u\in A$ and $v\in B$.
    Because $\alpha^2_{uv}, \alpha^3_{uv} \notin E(G)$, it is readily seen that $H$ cannot realize $M(K_4)$.
\end{proof}

We now move our attention to $\mathbb{Z}_2 \times \mathbb{Z}_2$.
We write $\alpha$, $\beta$, and $\gamma$ for the non-identity elements of $\mathbb{Z}_2\times\mathbb{Z}_2$.
For a subgraph $G$ of $K_n^{\mathbb{Z}_2\times\mathbb{Z}_2}$ and distinct integers $i, j \in [n]$, let $\ell(ij)$ be the set of elements $x\in \mathbb{Z}_2\times\mathbb{Z}_2$ such that $x_{ij}$ is an edge of $G$. As all elements in $\mathbb{Z}_2 \times \mathbb{Z}_2$ are there own inverses, $\ell(ij) = \ell(ji)$.
An \emph{$(a,b,c)$-triangle} is a triangle on three vertices $i,j,k$ such that $|\ell(ij)| \ge a$, $|\ell(ik)| \ge b$, and $|\ell(jk)| \ge c$.

The next two lemmas give sufficient conditions on a $\mathbb{Z}_2 \times \mathbb{Z}_2$-gain graph $G$ for $\FM(G)$ to contain a $M(K_4)$-submatroid.

\begin{lemma}\label{lem: 432}
    If $G$ is a $(4,3,2)$-triangle in $K_3^{\mathbb{Z}_2\times\mathbb{Z}_2}$, then $\FM(G)$ contains a $M(K_4)$-submatroid.
\end{lemma}
\begin{proof}
    By relabeling (and deleting edges if necessary), we may assume that $|\ell(12)| = 4$, $|\ell(13)| = 3$, and $|\ell(23)| = 2$.
    By switching at $2$, we may assume that $1 \in \ell(23)$. Suppose that $\ell(23) = \{1,x\}$; by switching at $1$ we may assume that $\{1,x\} \subseteq \ell(13)$. But then $E(G) \supseteq \{1_{12},1_{13},1_{23},x_{12},x_{13},x_{23}\}$ and so $\FM(G)$ contains $M(K_4)$ as a submatroid.
\end{proof}

\begin{lemma}\label{lem: 233332}
    Let $G$ be a gain subgraph of $K_4^{\mathbb{Z}_2\times\mathbb{Z}_2}$ such that the multiplicities of edges $|\ell(12)|, |\ell(34)| \ge 2$, and $|\ell(13)|, |\ell(14)|, |\ell(23)|, |\ell(24)| \ge 3$.
    Then $M(G)$ contains an $M(K_4)$-submatroid.
\end{lemma}
\begin{proof}
    By deleting edges if necessary, we may assume that all inequalities hold with equality.
    By switching-equivalence, we may assume that $\ell(13) = \ell(14) = \ell(23) = \{1,\alpha,\beta\}$.
    By symmetry, we may assume that $\ell(24)$ is $\{1,\alpha,\beta\}$, $\{1,\alpha,\gamma\}$, or $\{\alpha,\beta,\gamma\}$.
    In the last case, by switching at $1$ and $2$ by $\alpha$, we have that $\ell(13) = \ell(14) = \ell(23) = \{1,\alpha,\gamma\}$ and $\ell(24) = \{1,\gamma,\beta\}$, which can be reduced to the second case by an automorphism of $\mathbb{Z}_2\times\mathbb{Z}_2$.
    Therefore, we may assume that either $\ell(24) = \{1,\alpha,\beta\}$ or $\{1,\alpha,\gamma\}$.

    \emph{Case I: $\ell(24) = \{1,\alpha,\beta\}$.}
    Suppose that $1\in \ell(12)$.
    Then we may assume that $1\notin \ell(34)$.
    By symmetry, we may assume that $\alpha\in \ell(34)$.
    By switching at $3$ by $\alpha$, we deduce that $G$ contains $K_4$.
    Therefore, we may assume that $1\notin \ell(12)$ and similarly $1\notin \ell(34)$.
    If $\alpha\in \ell(12) \cap \ell(34)$, then by switching at $1$ and $3$ by $\alpha$, we see that $G$ contains $K_4$.
    By symmetry, we may assume that $\alpha\notin \ell(34) = \{\beta,\gamma\}$.
    Similarly, we may assume that $\beta \notin\ell(12) = \{\alpha,\gamma\}$.
    By switching at $1$ by $\gamma$, at $3$ by $\alpha$, and at $4$ by $\beta$, we deduce that $G$ contains $K_4$.

    \emph{Case II: $\ell(24) = \{1,\alpha,\gamma\}$.}
    Following a similar approach as in Case I, we can assume that $1\notin \ell(12)$ and $1\notin \ell(23)$.
    We can further assume that $\ell(12)=\{\alpha,\gamma\}$ and $\ell(34)=\{\beta,\gamma\}$.
    Then by switching at $1$ by $\alpha$ and at $4$ by $\gamma$, we see that $G$ contains $K_4$.
\end{proof}

\begin{theorem}
    $\frac{4}{7} \le \pi(\mathbb{Z}_2 \times \mathbb{Z}_2, M(K_4)) \le \frac{7}{12}$.
\end{theorem}

\begin{proof}
    The lower bound follows from Lemma~\ref{lem: K4 lower}.
    Thus, it remains  to show the upper bound.
    Let $M$ be a submatroid of $Q_n(\mathbb{Z}_2\times\mathbb{Z}_2)$ without $M(K_4)$-restriction, and
    let $G$ be a subgraph of $K_n^{\mathbb{Z}_2\times \mathbb{Z}_2}$ such that $\FM(G) = M$.
    We claim that
    \begin{equation}\label{eq: Z2xZ2 K4 main}
        E(G) \le \frac{7}{6}n^2 + 100n,
        \qquad\text{for all $n \ge 1$.}
    \end{equation}
    If $n\le 4$, then $\frac{7}{6}n^2 + 100n \ge 4 \binom{n}{2} + 4 = |Q_n(\mathbb{Z}_2\times\mathbb{Z}_2)|$, so~\eqref{eq: Z2xZ2 K4 main} holds trivially. We proceed by induction on $n\ge 5$.

    By Theorem~\ref{thm: Mantel Gamma}, we may assume that $M$ contains a $U_{2,3}$-submatroid.
    Then $G$ contains a balanced triangle, a theta-graph with two vertices, or a handcuff with two vertices. We show that $G$ contains a $(1,1,1)$-triangle; this is obvious in the case that $G$ contains a balanced triangle, so we may assume that $G$ contains a theta-graph or handcuff on two vertices. Let $u$ and $v$ be the two vertices.
    
    We may assume that there is a vertex $w\in V(G)\setminus \{u,v\}$ such that $w$ is incident with at least $5$ edges into $\{u,v\}$, since otherwise $G$ has at most
    \begin{equation*}
        |E(G\setminus\{u,v\})| + 4(n-2) + |E(G[\{u,v\}])|
        \le 
        \frac{7}{6}(n-2)^2+100(n-2) + 4(n-2) + 4\binom{2}{2}+2 < \frac{7}{6}n^2 + 100n
    \end{equation*}
    edges by the induction hypothesis.
    Then $G$ has a triangle because the multiplicity of parallel edges is at most $4$.

    The preceding argument shows that $G$ has a $(1,1,1)$-triangle.
    We denote its vertices by $u$, $v$, and $w$.
    We may assume that there is $x\in V(G)\setminus\{u,v,w\}$ such that $x$ is incident with at least $8$ edges into $\{u,v,w\}$, since otherwise $G$ has at most $\frac{7}{6}(n-3)^2 + 100(n-3) + 7(n-3) + 4\binom{3}{2}+3 < \frac{7}{6}n^2 + 100n$ edges by the induction hypothesis.
    
    Then $G$ contains a $(3,3,1)$-triangle or a $(4,2,1)$-triangle.
    Similarly, we may assume that there is a vertex outside of such a triangle such that it has at least $8$ edges into the triangle.
    Hence $G$ contains a $(3,3,2)$-, $(4,2,2)$-, or $(4,3,1)$-triangle.
    
    We relabel our vertices $u,v,w$ as the vertices of such a triangle.

    \emph{Case I: $\{u,v,w\}$ induces a $(3,3,2)$-triangle.}
    By relabelling we may assume that the multiplicities of edges $uv, vw, wu$ are at least $3,3,2$, respectively.
    We may assume that there is a vertex $x\in V(G)\setminus\{u,v,w\}$ having at least $8$ edges into the triangle $uvw$.
    By a straightforward case analysis and using Lemma~\ref{lem: 432}, we may assume that the multiplicities of $xu$, $xv$, $xw$ are $2$ or $3$.
    By Lemma~\ref{lem: 233332} and symmetry, we may assume that the multiplicities of $xu$, $xv$, $xw$ are $3$, $3$, $2$, respectively.
    Then $\{x,u,v\}$ induces a $(3,3,3)$-triangle.
    We may assume that there is a vertex $y\in V(G)\setminus\{x,u,v\}$ having at least $8$ edges into $\{x,u,v\}$.
    Then by Lemma~\ref{lem: 432}, the multiplicities of $yx$, $yu$, $yv$ are $2$ or $3$, and by Lemma~\ref{lem: 233332}, $M$ contains a $M(K_4)$-submatroid, a contradiction.

    \emph{Case II: $\{u,v,w\}$ induces a $(4,2,2)$-triangle.}
    By relabelling we may assume that the multiplicity of $uv$ is $4$.
    We can assume that there is a vertex $x \in V(G) \setminus\{u,v,w\}$ such that $x$ has at least $8$ edges into the triangle $\{u,v,w\}$.
    By Lemma~\ref{lem: 432}, we may assume that the multiplicity of $xw$ is 4 and the multiplicities of $xu$ and $xv$ are two.
    If each vertex $y \in V(G) \setminus \{u,v,w,x\}$ has at most $9$ edges into $\{u,v,w,x\}$, then $G$ has at most $|E(G\setminus\{u,v,w,x\})| + 9(n-4) + |E(G[\{u,v,w,x\}])| \le \frac{7}{6}(n-4)^2 + 100(n-4) + 9(n-4) + 4\binom{4}{2}+4 < \frac{7}{6}n^2 + 100n$ edges by the induction hypothesis.
    Thus, we can assume that there is $y$ that has at least $10$ edges into $\{u,v,w,x\}$.
    Then $G$ contains a $(4,3,2)$-triangle, and thus by Lemma~\ref{lem: 432}, $M$ has a $M(K_4)$-submatroid, a contradiction.

    \emph{Case III: $\{u,v,w\}$ induces a $(4,3,1)$-triangle.}
    By relabelling we may assume that the multiplicities of $uv, vw, wu$ are $4,3,1$, respectively.
    As before, we may assume that $G$ has a vertex $x$ outside the triangle such that $x$ has at least $8$ edges into the triangle.
    By Lemma~\ref{lem: 432} and simple analysis, we may assume that the multiplicities of $xu$ and $xw$ are at least $3$ and one of them is $4$.
    Thus, the vertices $u,x,w$ form a $(4,3,1)$-triangle.
    We may assume there is a vertex $y \in V(G) \setminus \{u,v,w,x\}$ such that $y$ is incident with at least $10$ edges into $\{u,v,w,x\}$, since otherwise $|E(G)| \le \frac{7}{6}(n-4)^2 + 100(n-4) + 9(n-4) + 4\binom{4}{2}+4 < \frac{7}{6}n^2 + 100n$. 
    By case analysis, Lemma~\ref{lem: 432}, and symmetry, we may assume that one of the following subcases holds:
    \begin{itemize}
        \item The multiplicities of $yu$ and $yw$ are $4$ and the multiplicities of $yv$ and $yx$ are at least $1$. In this case, by switching at $u$, $v$ and $w$ if necessary, it can be seen that $G[\{u,v,w,y\}]$ contains $K_4$, a contradiction.
        \item The multiplicities of $yu$ and $yw$ are at least $1$, and the multiplicity of $yv$ and $yx$ are $4$. In this case, by switching at $x$, $v$, and $y$ if necessary, it can be seen that $G[\{u,v,x,y\}]$ contains $K_4$, a contradiction.
        \item The multiplicities of $xu$, $xw$, $yu$, $yv$, $yw$, and $yx$ are at least $4$, $3$, $1$, $3$, $3$, and $3$, respectively. In this case, $\{v,w,y\}$ induces a $(3,3,3)$-triangle, and we have reduced the problem to Case I. \qedhere
    \end{itemize}
\end{proof}

\section{Extremal numbers of large cliques in \texorpdfstring{$Q_n(\mathbb{Z}_2)$}{Qn(Z2)}}\label{sec: excluding large Kt}

\subsection{The extremal number of \texorpdfstring{$M(K_4)$}{M(K4)} in \texorpdfstring{$Q_n(\mathbb{Z}_2)$}{Qn(Z2)}}

We explicitly compute $\ex(Q_n(\mathbb{Z}_2), M(K_4))$ and find the extremal matroids.

Let $H_{a,b}$ be a $\mathbb{Z}_2$-gain graph on the vertex set $[a+b]$ such that its edge set is 
\begin{align*}
    \{b_{i}: 1\le i\le a\} \cup 
    \{(-1)_{i,j}: 1\le i<j\le a\} \cup 
    \{1_{i,j}, -1_{i,j}: 1\le i\le a < j \le a+b\}.
\end{align*}
In $H_{a,b}$, the number of $3$-edge loose handcuffs is $\binom{a}{2}$, the number of $3$-edge tight handcuffs is $ab$, and the number of balanced triangles is $2\binom{a}{2}b$. Thus, the matroid $\FM(H_{a,b})$ has $\binom{a}{2} + a^2b$ triangles.

\begin{theorem}\label{thm: Z2 K4}
    $\ex(Q_n(\mathbb{Z}_2),M(K_4)) = \left\lfloor \frac{2}{3}n^2 + \frac{1}{3}n \right\rfloor$ for $n\ge 3$.
    The extremal matroids are the frame matroids induced by $K_{4}^{\mathbb{Z}_2} \setminus \{1_{1,2},1_{3,4},(-1)_{1,2},(-1)_{3,4}\}$ if $n = 4$ and induced by the two infinite families 
    \begin{itemize}
        \item $H_{n-m,m}$ with $m=\lfloor n/3 \rfloor$ for any $n \ge 3$,
        \item $H_{n-m,m}$ with $m=\lfloor n/3 \rfloor + 1$ if $3|(n-2)$.
    \end{itemize}
\end{theorem}

For each $k\ge 1$, the matroids $\FM(H_{2k+2,k})$ and $\FM(H_{2k+1,k+1})$ are non-isomorphic because they have different numbers of triangles.
Also, $\FM(H_{3,1})$ and $\FM(K_{4}^{\mathbb{Z}_2} \setminus \{1_{1,2},1_{3,4},(-1)_{1,2},(-1)_{3,4}\})$ are non-isomorphic because the latter matroid has four $U_{2,4}$-submatroids, while the former has none.

\begin{proof}
    The gain graph $H_{a,b}$ has no subgraph switching-isomorphic to any of the four graphs in Figure~\ref{fig:Z2 K4}.
    The number of edges in $H_{n-\lfloor n/3 \rfloor, \lfloor n/3 \rfloor}$ is $\left\lfloor \frac{2}{3}n^2 + \frac{1}{3}n \right\rfloor$.
    Thus, $\ex(Q_n(\mathbb{Z}),M(K_4)) \ge \left\lfloor \frac{2}{3}n^2 + \frac{1}{3}n \right\rfloor$.

    Next, we demonstrate the upper bound.
    Let $G$ be a $\mathbb{Z}_2$-gain subgraph of $K_n^{\mathbb{Z}_2}$ that has no subgraph switching-isomorphic to $K_4$, $K_3^1$, $H_3$, or $C_3^2$.
    It suffices to prove that $|E(G)| \le \left\lfloor \frac{2}{3}n^2 + \frac{1}{3}n \right\rfloor$.

    We proceed by induction on $n$.
    By simple analysis, we can show that the bound is tight for $n\in\{3,4,5\}$.
    Thus, we may assume that $n\ge 6$.
    By \Cref{thm: Mantel Z2}, we may assume that $M$ has a triangle, say $T$.
    Suppose that $G$ has no balanced triangle.
    Then $T$ corresponds to a theta graph or a (loose or tight) handcuff in $G$, on exactly two vertices, say $u$ and $v$.
    By the induction hypothesis, $G\setminus\{u,v\}$ has at most $\frac{2}{3}(n-2)^2 + \frac{1}{3}(n-2)$ edges.
    For every vertex $w$ other than $u,v$, there are at most two edges between $w$ and $\{u,v\}$ since $G$ has no balanced triangle.
    Then
    \begin{equation}
        |E(G)| 
        \le 
        \frac{2}{3}(n-2)^2 + \frac{1}{3}(n-2) + 2(n-2) + 4 
        =
        \frac{2}{3}n^2 - \frac{1}{3} n + 2
        < 
        \left\lfloor \frac{2}{3}n^2 + \frac{1}{3}n \right\rfloor,
        \label{eq: Z2 K4 I}
    \end{equation}
    a contradiction.
    Thus, we can assume that $T$ corresponds to a balanced triangle in $G$, and let $u,v,w$ be the vertices in such a triangle.
    Then $G\setminus\{u,v,w\}$ has at most $\frac{2}{3}(n-3)^2 + \frac{1}{3}(n-3)$ edges, and for each $x\in V(G)\setminus\{u,v,w\}$, it has at most four edges into $\{u,v,w\}$.
    Also, $\{u,v,w\}$ induces at most seven edges.
    Thus, 
    \begin{equation}
        |E(G)|
        \le 
        \frac{2}{3}(n-3)^2 + \frac{1}{3}(n-3) + 4(n-3) + 7
        =
        \frac{2}{3}n^2 + \frac{1}{3}n. 
        \label{eq: Z2 K4 II}
    \end{equation}
    As $|E(G)|$ is an integer, we conclude the desired upper bound.

    Finally, we show the extremal examples by induction on $n\ge 3$.
    Let $G$ be a subgraph of $K_n^{\mathbb{Z}_2}$ that has $\lfloor \frac{2}{3}n^2 + \frac{1}{3}n \rfloor$ edges and has no subgraph switching-isomorphic to $K_4$, $K_3^1$, $H_3$, or $C_3^2$.
    For $n=3$, it is readily checked that $G$ is switching-isomorphic to $H_{1,2}$, which has $7$ edges.
    Next, suppose that $n=4$.
    Then $G$ has $12$ edges, so $G$ is $K_4^{\mathbb{Z}_2}$ minus four edges.
    If $G$ has at most two loops, then $G$ has a balanced $\{C_3^2, H_3\}$-copy, a contradiction.
    One can check that $G$ is switching-isomorphic to $H_{3,1}$ if it has exactly three loops, and $G$ is switching-isomorphic to $K_{4}^{\mathbb{Z}_2} \setminus \{1_{1,2},1_{3,4},-1_{1,2},-1_{3,4}\}$ if it has four loops.
    Suppose that $n=5$. Then $G$ has $18$ edges, so $G$ is $K_5^{\mathbb{Z}_2}$ minus seven edges.
    If $G$ has at most two loops, then $G$ has a balanced $\{C_3^2, H_3\}$-copy, a contradiction.
    If $G$ has five loops, then $G$ has a balanced $K_3^1$-copy, a contradiction.
    In the remaining case that $G$ has $\ell$ loops with $\ell=3$ or $4$, one can see that $G$ is switching-isomorphic to $H_{\ell, 5-\ell}$.
    Therefore, we can assume that $n\ge 6$.

    By~\eqref{eq: Z2 K4 I}, $G$ has a balanced triangle, and we denote the vertices in such a triangle by $u,v,w$.
    In addition, the subgraph of $G$ induced by $\{u,v,w\}$ must have seven edges by~\eqref{eq: Z2 K4 II} and, hence, is switching-isomorphic to $H_{2,1}$.
    By relabelling, we can assume that $u,v$ has loops and $w$ has no loop.

    By the induction hypothesis, $G\setminus\{u,v,w\}$ is switching-isomorphic to $H_{n-2-\lfloor n/3 \rfloor, \lfloor n/3 \rfloor -1}$, $H_{n-3-\lfloor n/3 \rfloor, \lfloor n/3 \rfloor }$ if $2|(n-3)$, or $K_{4}^{\mathbb{Z}_2} \setminus \{1_{1,2},1_{3,4},-1_{1,2},-1_{3,4}\}$ if $n=7$.
    Let $X$ be the set of looped vertices and $Y$ be the set of non-looped vertices in $G\setminus\{u,v,w\}$.
    By~\eqref{eq: Z2 K4 II}, each vertex in $X\cup Y$ is incident with exactly four edges that are incident with $u$, $v$, or $w$.
    As $G$ has no banaced $K_3^1$-copy, there is exactly one edge incident with $x$ and $z$ for each $x\in X$ and $z\in \{u,v\}$.
    Also, each $x\in X$ is incident with two edges that are incident with $w$, and the triangle induced by $\{u,v,x\}$ is unbalanced.
    If $G\setminus\{u,v,w\}$ has two looped vertices $x$ and $x'$ such that the edge-multiplicity between $x$ and $x'$ is two, then the non-looped subgraph of $G$ induced by $\{w,x,x'\}$ is isomorphic to $C_3^2$, a contradiction.
    Therefore, $G\setminus\{u,v,w\}$ is not isomorphic to $K_{4}^{\mathbb{Z}_2} \setminus \{1_{1,2},1_{3,4},-1_{1,2},-1_{3,4}\}$.
    Then $Y\ne \emptyset$ and every pair of vertices $x\in X$ and $y\in Y$ has edge-multiplicity two.
    For each $y\in Y$, the edge-multiplicity into the vertices $u,v,w$ is either $(1,1,2)$ or $(2,2,0)$.
    If the first case occurs, then for any $x\in X$, the non-looped subgraph of $G$ induced by $\{w,x,y\}$ is isomorphic $C_3^2$.
    Thus, the second case holds.
    Therefore, $G$ is switching-isomorphic to the desired graph.
\end{proof}

\subsection{The extremal number of \texorpdfstring{$M(K_t)$}{M(Kt)}, \texorpdfstring{$t\ge 5$}{t >= 5} in \texorpdfstring{$Q_n(\mathbb{Z}_2)$}{Qn(Z2)}}

For $t \ge 5$, we compute $\ex(Q_n(\mathbb{Z}_2), M(K_t))$ explicitly in $n$ up to $O(t^3)$ error.

\begin{lemma}\label{lem: t-1 vertices}
    Let $t \ge 3$ and let $G$ be a $\mathbb{Z}_2$-gain graph on $t-1$ vertices such that it has a balanced $K_{t-1}$-copy and has neither balanced $K_{t-1}^1$-copy nor balanced $H_{t-1}$-copy.
    Then $|E(G)| \le (t-1)(t-2) + \lfloor \frac{t-2}{2} \rfloor$.
\end{lemma}
\begin{proof}
    Since $G$ has no balanced $K_{t-1}^1$-copy, $G$ has a non-looped vertex.
    By relabeling, we may assume that the vertex $1$ is non-looped.
    Since $G$ contains a balanced $K_{t-1}$-copy, we may assume that $1_{ij}$ is an edge of $G$ for all $1\le i<j\le t-1$.
    For each $2\le i \le t-1$, if the vertex $i$ is non-looped, then we attach a loop on~$i$ and delete the edge $(-1)_{1i}$ whenever it exists.
    This process does not decrease the number of edges, as well as it does not violate the condition that $G$ has a balanced $K_{t-1}$-copy and has no balanced $\{K_{t-1}^1, H_{t-1}\}$-copy.
    Therefore, we may assume that all vertices except for $1$ are looped.
    Since $G$ has no balanced $H_{t-1}$-copy, the number of non-loop edges incident with $i$ for each $2\le i\le t-1$ is strictly less than $2(t-2)$.
    Therefore,
    \begin{align*}
        |E(G)| \le 
        (t-2)
         + \frac{1}{2} 
         \big( 2(t-2)
         + (2(t-2)-1) (t-2) \big)
        = (t-1)(t-2) + \frac{t-2}{2}.
    \end{align*}
    Because $|E(G)|$ is an integer, we obtain the desired inequality.
\end{proof}

\begin{theorem}\label{thm: Z2 Kt}
        Let $t\ge 5$ and $n\ge 1$ be integers. Then
        \begin{align*}
            \ex(Q_n(\mathbb{Z}_2), M(K_t)) =
            \frac{t-2}{t-1}n^2 +  \left\lfloor\frac{t-2}{2} \right\rfloor \frac{1}{t-1} n + O(t^3).
        \end{align*}
\end{theorem}

\begin{proof}
    First, we show the upper bound, for which it suffices to show that if a gain subgraph $G$ of $K_n^{\mathbb{Z}_2}$ has no subgraph switching-isomorphic to $K_t$, $K_{t-1}^1$, or $H_{t-1}$, then
    \begin{equation}\label{eq: large clique upper bound}
        |E(G)| \le \frac{t-2}{t-1}n^2 +  \left\lfloor\frac{t-2}{2} \right\rfloor 
    \frac{n}{t-1} + (t-1)^3.
    \end{equation}
    By Theorem~\ref{thm: Z2 K4}, \eqref{eq: large clique upper bound} holds for $t=4$. We proceed by induction on $t \ge 4$ and $n \ge 0$.
    We may assume $n > (t-1)^2$ since otherwise $\frac{t-2}{t-1}n^2 +  \left\lfloor\frac{t-2}{2} \right\rfloor \frac{n}{t-1} + (t-1)^3 > n^2 = |Q_n(\mathbb{Z}_2)|$ and~\eqref{eq: large clique upper bound} holds trivially.

    Suppose that $G$ has a balanced $K_{t-1}$-copy, and let $U$ be its vertex set.
    Then $|U| = t-1$ and each vertex in $V(G)-U$ has at most $2(t-2)$ edges into $U$, since otherwise $G$ has a balanced $K_t$-copy.
    By induction and Lemma~\ref{lem: t-1 vertices},
    \begin{equation*}
        \begin{split}
            |E(G)|
            &\le 
            |E(G \setminus U)|
            + 
            2(t-2)(n-(t-1))
            +
            |E(G[U])| \\
            &\le
            \frac{t-2}{t-1}(n-(t-1))^2 + 
            \left\lfloor\frac{t-2}{2} \right\rfloor \frac{n-(t-1)}{t-1} + (t-1)^3 \\
            &\qquad\qquad+
            2(t-2)(n-(t-1))
            +
            (t-1)(t-2) + \left\lfloor \frac{t-2}{2} \right\rfloor \\
            &= 
            \frac{t-2}{t-1}n^2 +  \left\lfloor\frac{t-2}{2} \right\rfloor 
            \frac{n}{t-1} + (t-1)^3.
        \end{split}
    \end{equation*}

    We may therefore assume that $G$ does not contain a balanced $K_{t-1}$-copy, and inductively that $M(G)$ has an $M(K_{t-1})$-restriction.
    Let $U$ be a set of $t-2$ vertices that induce an $M(K_{t-1})$-restriction.
    Then $G[U]$ contains a balanced $K_{t-2}$-copy.
    Because $G$ has no balanced $K_{t-1}$-copy, each vertex in $V(G)-U$ has at most $2(t-3)$ edges into $U$, so, by induction,
    \begin{equation*}
        \begin{split}
            |E(G)|
            &\le 
            |E(G \setminus U)|
            + 
            2(t-3)(n-(t-2))
            +
            |E(G[U])| \\
            &\le
            \frac{t-2}{t-1}(n-(t-2))^2 +  
            \left\lfloor\frac{t-2}{2} \right\rfloor 
            \frac{n-(t-2)}{t-1} + (t-1)^3 \\
            &\qquad\qquad+
            2(t-3)(n-(t-2))
            +
            (t-2)^2 \\
            &<
            \frac{t-2}{t-1}n^2 +  \left\lfloor\frac{t-2}{2} \right\rfloor 
            \frac{n}{t-1} + (t-1)^3,
        \end{split}
    \end{equation*}
    where the last inequality holds because $n > (t-1)^2$.

    To prove a matching lower bound, we construct a dense gain subgraph $G$ of $K_n^{\mathbb{Z}_2}$ that induces a $M(K_t)$-free matroid; the construction depends on the parity of $t$.
    Let $q$ and $r$ be the integers such that $n = q(t-1) + r$ with $0\le r < t-1$.
    For $i=1,2,\ldots,t-2$, let $A_i = \{q(i-1)+j : 1 \le j \le q\}$, and let $A_{t-1} = [n] \setminus \bigcup_{i=1}^{t-2} A_i$. For $j=1,2,\ldots,\lfloor\frac{t-2}{2}\rfloor$, let $B_j = A_{2j-1} \cup A_{2j}$.

    When $t$ is even, let $G$ be the subgraph of $K_n^{\mathbb{Z}_2}$ obtained by restricting to the edge set
    \begin{equation*}
        E(G) = E(K_n^{\mathbb{Z}_2}) \setminus \left(\bigcup_{j=1}^{\frac{t-2}{2}} \{1_{st} : s,t \in B_j\} \bigcup E(K_n^{\mathbb{Z}_2}[A_{t-1}])\right),
    \end{equation*}
    so
    \begin{equation*}
        |E(G)| = 
        n^2 - \binom{2q}{2} \cdot \frac{t-2}{2} - (q+r)^2 \ge
        \frac{t-2}{t-1}n^2 + \frac{t-2}{2(t-1)}n - 10 t^2.
    \end{equation*}
    We show that $G$ has no balanced $\{K_t, K_{t-1}^1, H_{t-1}\}$-copy.
    Suppose that $G$ has a balanced $K_t$-copy, say $X$.
    Each triple of vertices in $B_j$ induces an unbalanced triangle, so $|V(X) \cap B_j| \le 2$ for each $1\le j\le \frac{t-2}{2}$.
    Since $G[A_{t-1}]$ spans no edges, $|V(X) \cap A_{t-1}| \le 1$.
    Then $|V(X)| \le t-1$, a contradiction.
    Suppose that $G$ has a balanced $K_{t-1}^1$-copy, say $Y$. As before, $|V(Y) \cap B_j| \le 2$ for each $1 \le j \le \frac{t-2}{2}$.
    As the vertices in $A_{t-1}$ span no loops, $V(Y) \cap A_{t-1} = \emptyset$, and we again reach a contradiction.
    Finally, if $G$ has a balanced $H_{t-1}$-copy, say $Z$ with center vertex $z$,
    then $z \notin A_{t-1}$, and we may assume that $z\in B_1$.
    It is easily shown that $|V(Z) \cap B_j| = 2$ for each $1\le j\le \frac{t-2}{2}$, and let $z'$ be the vertex in $V(Z) \cap B_1$ other than $z$.
    Then $1_{zz'}\notin E(G)$, a contradiction.

    When $t$ is odd, let $G$ be the subgraph of $K_n^{\mathbb{Z}_2}$ obtained by restricting to the edge set
    \begin{equation*}
        E(G) = E(K_n^{\mathbb{Z}_2}) \setminus \left(\bigcup_{j=1}^{\frac{t-3}{2}} \{1_{st} : s,t \in B_j\} \bigcup E(K_n^{\mathbb{Z}_2}[A_{t-2}]) \bigcup E(K_n^{\mathbb{Z}_2}[A_{t-1}])\right).
    \end{equation*}
    We omit the proof that $G$ has no balanced $\{K_t, K_{t-1}^1, H_{t-1}\}$-copy and the analysis of its number of edges, which is similar to the case when $t$ is even.
\end{proof}

We note that each of the graphs constructed in the proof of Theorem~\ref{thm: Z2 Kt} contains a copy of the Tur\'{a}n graph $T_{n,t-1}$, each of whose edges is labelled~1. When $n$ is divisible by $t-1$, the number of edges in these graphs is exactly $\frac{t-2}{t-1}n^2 +  \left\lfloor\frac{t-2}{2} \right\rfloor \frac{1}{t-1} n$. We ask if this construction is the best possible.

\section*{Acknowledgements}

This project started when Van der Pol visited the Discrete Mathematics Group at the Institute for Basic Science, and he thanks IBS for its hospitality. We thank Daryl Funk for discussions regarding gain-graphs and Hyonwoo Lee for discussions regarding extremal numbers. Campbell and Kim were supported by IBS grant IBS-R029-C1.

\bibliography{ref.bib}

\appendix

\section{Matchstick and origami geometries}\label{sec: matchstick and origami}

In this section, we consider the extremal numbers for matchstick and origami geometries. Kahn and Kung~\cite{KK1982} showed that these are the only examples of sequences of universal models for matroid varieties, other than projective geometries and Dowling geometries.

\subsection{Matchstick geometries}

The even-rank \emph{matchstick geometry} $M_{2k}(n)$ is isomorphic to $U_{2,n+1}^{\oplus k}$, the direct sum of $k$ $(n+1)$-point lines, while the odd-rank matchstick geometry $M_{2k+1}(n)$ is isomorphic to $M_{2k}(n) \oplus U_{1,1}$. We note that $M_t(m)$ is a submatroid of $M_r(n)$ if and only if $t\le r$ and $m\le n$.

\begin{proposition}\label{prop: matchstick}
    Let $n \ge m \ge 1$ and $r \ge t \ge 1$. Let $k := \lfloor \frac{r}{2} \rfloor$ and $\ell := \lfloor \frac{t}{2} \rfloor$.
    \begin{itemize}
        \item 
        If $r$ is even, then $\ex(M_{r}(n),M_{t}(m)) = (\ell-1)(n+1)+(k-\ell+1)m$; the extremal matroid is isomorphic to $U_{2,n+1}^{\oplus \ell-1} \oplus U_{2,m}^{\oplus k - \ell + 1}$.
        \item 
        If $r$ is odd and $t < r$, then $\ex(M_{r}(n),M_{t}(m)) = (\ell-1)(n+1)+(k-\ell+1)m + 1$; the extremal matroid is isomorphic to $U_{2,n+1}^{\oplus \ell-1} \oplus U_{2,m}^{\oplus k-\ell+1} \oplus U_{1,1}$.
        \item 
        If $r$ is odd and $t = r$, then $\ex(M_{r}(n),M_{t}(m)) = \ell(n+1)$; the extremal matroid is isomorphic to $M_{r-1}(n) \cong U_{2,n+1}^{\oplus \ell}$.
    \end{itemize}
\end{proposition}

\begin{proof}
    We first deal with the case that both $r$ and $t$ are even, so $M_r(n) \cong U_{2,n+1}^{\oplus k}$ and $M_t(m) \cong U_{2,m}^{\ell}$.
    Let $N$ be a submatroid of $M_r(n)$ without $M_t(m)$-submatroid.
    If $|N| > (\ell-1)(n+1) + (k-\ell+1)m$, then by the pigeonhole principle, $N$ contains $\ell$ disjoint copies of $U_{2,m+1}$ and therefore $N$ contains a $M_t(m)$-submatroid, a contradiction. Hence, we have $|N| \le (\ell-1)(n+1) + (k-\ell+1)m$. It is straightforward to see that if $|N| = (\ell-1)(n+1) + (k-\ell+1)m$, then $N\cong  U_{2,n+1}^{\oplus \ell-1} \oplus U_{2,m}^{\oplus k-\ell+1}$.

    The remaining cases are proved similarly.
\end{proof}

\subsection{Origami geometries}

The \emph{origami geometry} $O_r(n)$ is a rank-$r$ matroid obtained from a free matroid on $r$ elements, $p_1,p_2,\dots,p_{r}$, by adding $n-1$ points freely on the line joining $p_i$ and $p_{i+1}$ for each $i=1,\dots,r-1$.
Note that $O_r(1) \cong U_{r,r}$, $O_1(n) \cong U_{1,1}$, $O_2(n) \cong U_{2,n+1}$, and $O_r(n)$ has $(r-1)n+1$ points.
The following two propositions are straightforward from the definition.

\begin{proposition}\label{prop: origami1}
    For any $n$ and $r\ge t$, we have $\ex(O_r(n), U_{t,t}) = (t-2)n+1$ and 
    the extremal matroid is $O_{t-1}(n)$.
    \qed
\end{proposition}

\begin{proposition}\label{prop: origami2}
    For any $n\ge m$ and $r$, we have $\ex(O_r(n), U_{2,m+1}) = (r-1)(m-1)+1$ and
    the extremal matroid is $O_r(m-1)$.
    \qed
\end{proposition}

The below proposition together with the previous two results gives the full characterization of the extremal numbers $\ex(O_r(n),O_t(m))$.

\begin{proposition}\label{prop: origami3}
    For $n > m \ge 2$ and $r \ge t \ge 3$, we have
    $\ex(O_r(n), O_t(m))
        =
        (r-1)n + 1 - \left\lfloor \frac{r-2}{t-2} \right\rfloor$.
    For $n = m \ge 2$ and $r \ge t \ge 3$, we have
    $\ex(O_r(n), O_t(n))
        =
        (r-1)n + 1 - \lfloor r/t \rfloor$.
\end{proposition}

\begin{proof}
    We first handle the case that $n=m$.
    Let $X\subseteq E(O_r(n))$ be a set such that $O_r(n) \setminus X$ is $O_t(m)$-free.
    We claim that $|X| \ge \lfloor r/t \rfloor$.
    For each point $x \in X \setminus \{ p_1,\dots,p_{r} \}$, we replace $x$ with $p_i$ or $p_{i+1}$, where the line spanned by $p_i$ and $p_{i+1}$ contains $x$. If we denote the resulting set by $X'$, then $|X'| \le |X|$ and $O_r(n) \setminus X'$ is $O_t(m)$-free.
    Therefore, we may assume that $X \subseteq \{p_1,\ldots,p_{r}\}$.
    As $O_r(n) \setminus X$ is $O_t(m)$-free, there are no $t$ consecutive points $p_{i+1},\ldots,p_{i+t}$ in $\{p_1,\ldots,p_r\} \setminus X$, which shows $|X| \ge \lfloor r/t \rfloor$.
    It is readily shown that  $O_r(n) \backslash \{ p_{kt} : 1 \le k \le \lfloor r/t \rfloor \}$ is $O_t(n)$-free. Thus, $\ex(O_r(n), O_t(m)) = (r-1)n + 1 - \lfloor r/t \rfloor$.

    Next, we deal with the case that $n>m$.
    The proof is almost identical to that of the previous case.
    Let $X\subseteq E(O_r(n))$ be a set such that $O_r(n) \backslash X$ is $O_t(m)$-free.
    We claim that $|X| \ge \lfloor \frac{r-2}{t-2} \rfloor$.
    For each point $x \in X \setminus \{ p_1,\dots,p_{r} \}$, we replace $x$ with $p_i$ or $p_{i+1}$, where the line spanned by $p_i$ and $p_{i+1}$ contains $x$. However, unlike the previous case $n=m$, we additionally require that $x$ be replaced by $p_{i+1}$ if $p_i$ is already in $X$, and by $p_i$ otherwise.
    Then the resulting set $X'$ satisfies $|X'| \le |X|$ and $O_r(n) \setminus X'$ is $O_t(m)$-free.
    Therefore, we may assume that $X \subseteq \{p_1,\ldots,p_{r}\}$.
    Since $O_r(n) \setminus X$ is $O_t(m)$-free, we deduce that $\{p_{i+1},\ldots,p_{i+t-2}\} \cap X \ne \emptyset$ for each $1 \le i \le r-t+1$, which implies that $|X| \ge \lfloor \frac{r-2}{t-2} \rfloor$.
    One can check that $O_r(n) \setminus \{ p_{k(t-2)+1} : 1 \le k \le \lfloor \frac{r-2}{t-2} \rfloor \}$ is $O_t(m)$-free. Therefore, $\ex(O_r(n), O_t(m)) = (r-1)n + 1 - \lfloor \frac{r-2}{t-2} \rfloor$.
\end{proof}

\end{document}